\documentclass{article}
\usepackage[a4paper, margin=2.7cm]{geometry}
\usepackage{amsmath,amssymb,amsfonts,amsthm,tikz,booktabs,stmaryrd,cite,float}
\usetikzlibrary{positioning}
\usetikzlibrary{calc}
\usetikzlibrary{arrows.meta}
\usepackage[hidelinks]{hyperref}

\newtheorem{theorem}{Theorem}
\newtheorem{proposition}{Proposition}%
\newtheorem{lemma}{Lemma}%
\newtheorem{corollaryC}{Corollary of Conjecture}
\newtheorem{remark}{Remark}%
\newtheorem{definition}{Definition}

\newtheorem{conjecture}{Conjecture}

\newcommand{\Ttran}{\mathsf{T}}

\newcommand{\fro}{\mathsf{F}}

\newcommand{\defi}{:=}
\newcommand{\van}{\mathsf{Van}}

\newcommand{\poly}{\mathbb{P}}
\newcommand{\chim}{\chi_{\mathsf{mono}}}
\newcommand{\chic}{\chi_{\mathsf{coef}}}
\newcommand{\cl}{(\mathsf{cl})}

\newcommand{\fp}{F}
\newcommand{\R}{\mathbb{R}}

\newcommand{\rk}{\mathrm{rank}}
\newcommand{\range}{\mathsf{range}}

\newcommand{\gap}{\mathsf{relgap}}

\newcommand{\cheb}{\mathsf{Cheb}}
\newcommand{\fig}{eps}

\newcommand{\figsizeD}{0.45\textwidth}

\providecommand{\diagm}{\mathsf{diag}}
\providecommand{\diagM}[1]{\mathsf{diag}(#1)}
\providecommand{\abs}[1]{\lvert#1\rvert}
\providecommand{\norm}[1]{\lVert#1\rVert}

\providecommand{\bigabs}[1]{\bigl\lvert#1\bigr\rvert}

\providecommand{\bignorm}[1]{\bigl\lVert#1\bigr\rVert}
\providecommand{\Bignorm}[1]{\Bigl\lVert#1\Bigr\rVert}

\newtheorem*{customthm}{\customthmname}
\newenvironment{namedthm}[1]
  {\newcommand{\customthmname}{#1}\begin{customthm}}
  {\end{customthm}}
\usepackage[capitalise,nameinlink,noabbrev]{cleveref}
\crefname{assumption}{Assumption}{Assumptions}
\crefname{conjecture}{Conjecture}{Conjectures}
\crefname{lemma}{Lemma}{Lemma}
\crefname{corollaryC}{Corollary of Conjecture}{}
\crefname{proposition}{Proposition}{Propositions}
\crefname{equation}{}{}

\begin{document}

\title{A structural bound for cluster robustness of randomized small-block Lanczos}
\author{
    Nian Shao\thanks{
        Institute of Mathematics, EPF Lausanne, 1015 Lausanne, Switzerland
        (\href{mailto:nian.shao@epfl.ch}{nian.shao@epfl.ch}).
    }
}

\maketitle

\begin{abstract}
    The Lanczos method is a fast and memory-efficient algorithm for solving large-scale symmetric eigenvalue problems. However, its rapid convergence can deteriorate significantly when computing clustered eigenvalues due to a lack of cluster robustness. A promising strategy to enhance cluster robustness---without substantially compromising convergence speed or memory efficiency---is to use a random small-block initial, where the block size is greater than one but still much smaller than the cluster size. This leads to the Randomized Small-Block Lanczos (RSBL) method. Despite its empirical effectiveness, RSBL lacks the comprehensive theoretical understanding already available for single-vector and large-block variants. In this paper, we develop a structural bound that supports the cluster robustness of RSBL by leveraging tools from matrix polynomials. We identify an intrinsic theoretical challenge stemming from the non-commuting nature of matrix multiplication. To provide further insight, we propose a conjectured probabilistic bound for cluster robustness and validate it through empirical experiments. Finally, we discuss how insights into cluster robustness can enhance our understanding of RSBL for both eigenvalue computation and low-rank approximation.
\end{abstract}
\section{Introduction}
Given a symmetric matrix $A\in \R^{n\times n}$ and an initial $\Omega\in\R^{n\times b}$, the $\ell$th block Krylov subspace is generated by successive applications of $A$ to $\Omega$:
\begin{equation*}
    \mathcal{K}_{\ell}(A,\Omega) \defi \range[\Omega,A\Omega,\dotsc,A^{\ell-1}\Omega].
\end{equation*}
With an orthonormal basis computed by the block Lanczos process, the block Lanczos method then performs Rayleigh--Ritz extraction
from the resulting block Krylov subspace.
The computed approximations of eigenvalues and eigenvectors of $A$ are called Ritz values and Ritz vectors.

Compared to subspace iteration, a particularly attractive feature of Krylov subspace methods is their ability to extract significantly more Ritz pairs than the block size $b$.
This enables a high-degree polynomial acceleration while requiring only a moderate number of matrix-vector multiplications (matvecs) with $A$. In fact, in most modern scientific computing environments---such as Matlab, Python, and Julia---the default large-scale eigensolvers (e.g., \texttt{eigs} in Matlab) are based on a single-vector ($b = 1$) Arnoldi/Lanczos method, usually with implicit restarting \cite{Lehoucq1998,wu2000thick}, even when many eigenvalues are sought.

A known limitation of the single-vector Lanczos is its lack of \emph{cluster robustness}. Specifically, when computing a cluster of closely spaced eigenvalues, the single-vector Lanczos method suffers from slow convergence due to small spectral gaps.
A particularly extreme and notorious example is null space computation \cite{kressner2024randomized}, where single-vector Lanczos---and even small-block variants with very small block sizes---may fail to correctly determine the null space dimension.
Indeed, whenever two eigenvalues of interest are separated by a tiny gap, the convergence of single-vector Lanczos method becomes significantly slower \cite{Meyer2023,kressner2024randomized}.
Unfortunately, in many practical eigenvalue problems, clustered or multiple eigenvalues frequently arise, often as a consequence of symmetries or periodic structures inherent in physical models \cite{seyranian1994multiple,bathe1973solution}.
In such cases, a single-vector Lanczos method becomes less favorable as the modest gains in convergence speed are insufficient to offset the high communication costs~\cite{Ballard2014}, particularly in large-scale parallel computing environments.  
Yet even on a personal computer, multiplying a $1000\times 1000$ (dense) Gaussian random matrix by $16$ vectors in Matlab takes only about $3.3\times$ the time needed for one matvec. 
Often, as noted in \cite[Sec.~13.12]{Parlett1998}, it is suggested to choose a large block size $b$ that is at least as large as the size of the eigenvalue cluster of interest.  This strategy, which we refer to as the large-block Lanczos method, resolves the challenges posed by clustered eigenvalues and can be implemented in a communication-avoiding manner \cite[Sec.~1.6.2]{Hoemmen2010}.
However, it typically requires significantly more matvecs and substantially more memory than the single-vector Lanczos method, particularly when the eigenvalue cluster is large and the restarting is employed.

One strategy to find a compromise between the cluster robustness and low communication cost of large-block Lanczos with the fast convergence and memory efficiency of the single-vector variant is to employ a (Randomized) Small-Block Lanczos (RSBL) method, where the block size is greater than one but still significantly smaller than the cluster size. It is clear that RSBL has better communication cost than single-vector Lanczos and lower memory requirement than the large-block variant. The following example shows that, in fact, unless we are in an extreme situation where eigenvalues are of high multiplicity, such as null space computation, RSBL is already robust enough. 
Let $A = \diagm\{\Lambda,\Lambda_{\perp}\}\in\R^{2000\times 2000}$, where $\Lambda\in\R^{32\times 32}$ and $\Lambda_{\perp}$ are diagonal matrices containing eigenvalues uniformly spaced in $[1,1+\beta]$ and $[-1,0]$, respectively, and $\beta$ is a parameter to control the radius of the cluster. We run RSBL with Gaussian random initial matrix $\Omega\in\R^{2000\times b}$ to compute all eigenvalues in $\Lambda$. When executing block Lanczos with full reorthogonalization in double precision, and declaring the error of all Ritz values less than $10^{-10}$ as the convergence criterion, we obtain the results reported in~\cref{fig:intro}.

\begin{figure}[htbp]
    \centering
    \includegraphics[width=\figsizeD]{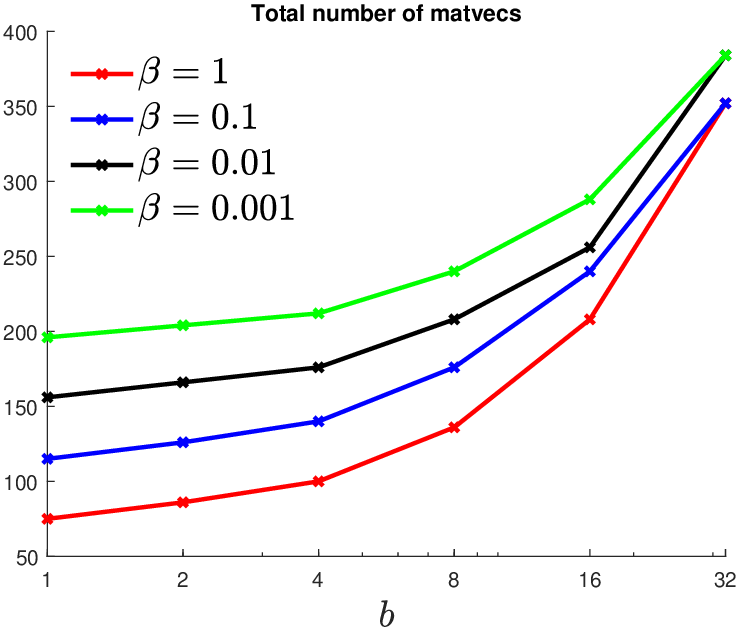}
    \includegraphics[width=\figsizeD]{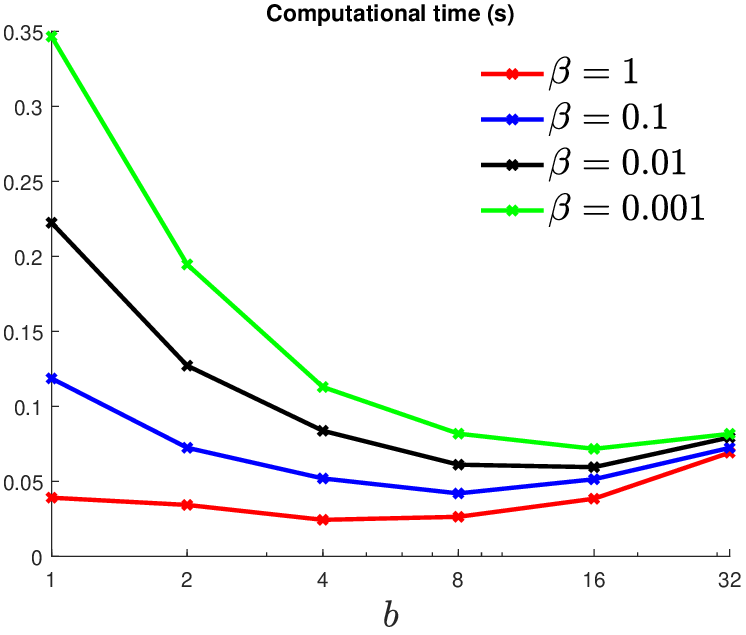}
    \caption{Total number of matvecs and computational time by block Lanczos required for the convergence of 32 eigenvalues in a cluster with $\beta$ being the cluster diameter.}
    \label{fig:intro}
\end{figure}

While the number of matvecs required by the single-vector Lanczos method is consistently lower than that of block variants, the relative overhead incurred by RSBL quickly becomes negligible as the target eigenvalues form a tighter cluster. When communication costs are considered, RSBL will achieve a lower overall computational time.
Other illustrative examples that demonstrate the improved cluster robustness of small-block initials include the block Sakurai--Sugiura method \cite{Ikegami2010} and the block variant of Rayleigh--Ritz method with contour integrals \cite{Ikegami2010rr} for computing interior eigenvalues.

Despite the empirical effectiveness of RSBL, there exists surprisingly limited theoretical understanding, particularly regarding their robustness to eigenvalue clusters. Most existing analyses of (randomized) Lanczos methods for eigenvalue computation or low-rank approximation focuses on the single-vector case \cite{kressner2024randomized,Meyer2023,Saad1980,Kaniel1966} or the large-block case \cite{Tropp2022,Li2015,Musco2015,Drineas2018,persson2025randomized,tropp2023randomized}. 
Let $A\in\R^{n\times n}$ be a symmetric matrix with spectral decomposition
    \begin{equation*}
        A = \begin{bmatrix}
            Q & Q_{\perp}
        \end{bmatrix} 
        \begin{bmatrix}
            \Lambda&\\ 
            &\Lambda_{\perp}
        \end{bmatrix}
        \begin{bmatrix}
            Q & Q_{\perp}
        \end{bmatrix}^{\Ttran},
        \quad \text{where}\quad \Lambda = \diagm\{ \Lambda_{1},\dotsc,\Lambda_{d} \}
    \end{equation*}
    with $\Lambda_{i}\in\R^{b\times b}$ for $1\leq i\leq d$. Given the block size $b$ of $\Lambda_{i}$, we define the \emph{$b$th order (relative) eigenvalue gap} by 
    \begin{equation}
        \label{eq:defgap}
    \gap_{b} \defi \min_{\substack{1\leq i\neq j\leq d\\     
    \lambda_{i}\in \Lambda_{i}, \lambda_{j}\in \Lambda_{j}\\ }}
    \frac{\abs{\lambda_{i}-\lambda_{j}}}{\lambda_{\max}-\lambda_{\min}},
\end{equation}
where $\lambda_{\min}$ and $\lambda_{\max}$ are the smallest and largest eigenvalues of $A$, respectively. 
In \cite[Thm.~4.5]{Meyer2023}, a theoretical analysis establishes that the convergence of RSBL depends on $\gap_{b}$. This property enables RSBL to effectively approximate matrices with multiple leading singular values, provided their multiplicity does not exceed $b$.
However, this result indicates that matvec complexity of RSBL scales linearly with the block size $b$, which does not fully justify its observed effectiveness. This gap in understanding motivates our investigation in this direction.
Independently of our work, \cite{Chen2026} analyzes RSBL for low-rank approximations; we provide a detailed discussion of their results at the end of this section.

In this paper, we develop a new approach for analyzing the RSBL on eigenvalue problems. 
Following the structure of \cite{Meyer2023,kressner2024randomized}, 
we quantify cluster robustness using the metric
\begin{equation}
\label{eq:defangle}
\tan\angle (\range(Q),\mathcal{K}_{d}(A,\Omega)),
\end{equation}
where $\angle$ is the largest principal angle.
This quantity reduces to  \cite[Eq.~(3.5)]{kressner2024randomized} when $b=1$, which plays an important role in the convergence of (randomized) single-vector Lanczos method for eigenvalue problems.
To obtain an upper bound, we adopt tools from matrix polynomials developed in \cite{Dennis1976} and derive an explicit interpolation formula for matrix polynomials on certain structured matrices. This result generalizes classical Lagrange interpolation for scalar polynomials, which has been crucial in the analysis of randomized single-vector Lanczos methods.
Under mild spectral gap assumptions---specifically, the multiplicity of desired eigenvalues does not exceed $b$---we provide the following structural bound for \cref{eq:defangle} when $\Omega$ is a Gaussian random matrix. 

\begin{namedthm}{Abridged of \cref{thm:sb}}
    Let $\Omega\in\R^{n\times b}$ be a Gaussian random matrix. Assume that $\gap_{b}>0$, then, with high probability, 
\begin{equation*}
    \tan\angle\bigl(\range(Q),\mathcal{K}_{d}(A,\Omega)\bigr) \leq c_{\Omega}\cdot \Bigl(\frac{\chim\chic}{\gap_{b}}\Bigr)^{d-1},
\end{equation*}
where $c_{\Omega}$ is a constant depending on $n,b,d$ and the failure probability, $\chim$ and $\chic$ are quantities to be specified later in \cref{eq:defchi}.
\end{namedthm}

However, the presence of non-commuting products of random matrices, an intrinsic challenge in the analysis of RSBL, prevents us from obtaining a rigorous probabilistic bound for $\chim$ and $\chic$.
Instead, we propose a conjectured probabilistic bound, and support this conjecture with empirical numerical evidence.
Finally, we discuss how our results contribute to a deeper understanding of RSBL for eigenvalue computation and low-rank approximation.

\begin{remark}
    \label{rmk:intro}
The quantity in \cref{eq:defangle} is not expected to be small, as it scales with the notoriously high condition numbers of Vandermonde-type matrices. Nevertheless, this does not hinder the effectiveness of RSBL. Indeed, for an $\ell\gg d$, the following bound holds:
\begin{equation*}
    \tan \angle(\range(Q),\mathcal{K}_{\ell}(A,\Omega)) \leq c\gamma^{\ell-d}\tan\angle(\range(Q),\mathcal{K}_{d}(A,\Omega)),
\end{equation*}
where $c>0$ and $0<\gamma<1$ are constants independent of $\ell$. This illustrates the exponential convergence of RSBL; see \cref{sec:app} for details.
\end{remark}

\paragraph{Relation to \cite{Chen2026}.}
Using a fundamentally different approach based primarily on matrix anti-concentration results from \cite{Nie2022}, a cluster robustness-type result is established in \cite[Thm.~4.1]{Chen2026} for scenarios where \emph{all $b\cdot d$ target eigenvalues are well-separated}, that is, the standard (relative) eigenvalue gap $\gap_{1}$ is non-negligible.
However, in \cref{sec:numexp}, empirical results demonstrate that cluster robustness depends on $\gap_{b}$ rather than $\gap_{1}$.
While \cite{Chen2026} also \emph{conjectures} that $\gap_{1}$ can be replaced by $\gap_{b}$, proving this will likely require techniques different from those in \cite{Chen2026}.
In contrast, in our framework, the $\gap_{b}$ arises naturally from the interpolation of matrix polynomials; see \cref{thm:interpolation} for further details.

\paragraph*{Notation.}
We use $\norm{\cdot}$ to denote the Euclidean norm of a vector and the spectral norm of a matrix. For a family of (non-commuting) matrices $A_{1},\dotsc,A_{d}$, we denote $\prod_{i=1}^{d}A_{i} = A_{1}A_{2}\dotsb A_{d}$.

\section{Matrix polynomials}
\label{sec:mplm}
To analyze RSBL, we employ techniques involving matrix polynomials, as introduced in \cite{Dennis1976}. These mathematical tools play an important role in the analysis of systems and control theory \cite{Sontag1998, Lancaster2002}. In the context of block Krylov subspace methods, they are also used in the analysis of linear system solvers \cite{Kubinova2020, Simoncini1996} and computation of matrix functions \cite{Frommer2020}.

\begin{definition}[Matrix polynomial]
    Let $C_{0},\dotsc,C_{d},X\in\R^{b\times b}$. Then we define a matrix polynomial\footnote{When $X=\lambda I$ for $\lambda\in\R$ is a scalar matrix, the matrix polynomial is also referred to as a lambda-matrix.} of degree at most $d$ by $\Phi: \R^{b\times b}\mapsto \R^{b\times b}$ as 
    \begin{equation}
        \label{eq:defPhiX}
        \Phi(X) = C_{0}+XC_{1}+\dotsb+X^{d}C_{d}.
    \end{equation} 
    We denote the subspace that contains all such matrix polynomials of degree at most $d$ by $\poly_{d}(\R^{b\times b})$.
\end{definition}

We remind readers that in our definition of the matrix polynomial~\cref{eq:defPhiX}, the coefficients $C_{i}$ appear to the right of the variable $X$, which differs from the convention in \cite{Dennis1976}.

\subsection{Upper bound on $\norm{\Phi(B)}$}
The first result we develop is an upper bound of the spectral norm of $\Phi(B)$ for diagonalizable $B$ with real eigenvalues.
\begin{lemma}
    \label{lem:boundMP}
    Suppose that $B\in\R^{b\times b}$ admits an eigenvalue decomposition $B = \Omega^{-1} \Lambda \Omega$, where $\Lambda$ is a diagonal matrix containing real eigenvalues $\lambda^{(1)}\leq \dotsb \leq \lambda^{(b)}$. For any matrix polynomial $\Phi\in\poly_{d}(\R^{b\times b})$, 
    \begin{equation*}
        \norm{\Phi(B)}\leq \sqrt{b}\norm{\Omega^{-1}}\norm{\Omega}\max_{\lambda^{(1)}\leq \lambda \leq \lambda^{(b)}}\norm{\Phi(\lambda I)}.
    \end{equation*}
\end{lemma}
\begin{proof}
    By definition, $\Phi$ takes the form \cref{eq:defPhiX}. For any $x\in\R^{b}$, let $\varphi_{j}(\lambda) = \sum_{i=0}^{d} (e_{j}^{\Ttran}\Omega C_{i}x)\lambda^{i}$ be a scalar polynomial, where $1\leq j\leq b$. By submultiplicativity, 
    \begin{equation*}
        \begin{aligned}
            \norm{\Phi(B)x}^{2} &= \Bignorm{\sum_{i=0}^{d}B^{i}C_{i}x}^{2} = \Bignorm{\Omega^{-1}\sum_{i=0}^{d}\Lambda^{i}\Omega C_{i}x}^{2} 
            \leq \norm{\Omega^{-1}}^{2}\Bignorm{\sum_{i=0}^{d}\Lambda^{i}\Omega C_{i}x}^{2}  
            = \norm{\Omega^{-1}}^{2}\sum_{j=1}^{b}\varphi_{j}^{2}(\lambda^{(j)})  \\ 
            &\leq b\norm{\Omega^{-1}}^{2}\max_{1\leq j\leq b}\varphi_{j}^{2}(\lambda^{(j)}) 
            \leq b\norm{\Omega^{-1}}^{2}\max_{\substack{1\leq j\leq b\\ \lambda^{(1)}\leq \lambda \leq \lambda^{(b)}}} \varphi_{j}^{2}(\lambda)  
            \leq b\norm{\Omega^{-1}}^{2}\max_{\lambda^{(1)}\leq \lambda \leq \lambda^{(b)}}\sum_{j=1}^{b}\varphi_{j}^{2}(\lambda).
        \end{aligned}
    \end{equation*}
    Note that
    \begin{equation*}
        \sum_{j=1}^{b}\varphi_{j}^{2}(\lambda) = \Bignorm{\sum_{i=0}^{d}\lambda^{i}\Omega C_{i}x}^{2} = \norm{\Omega\Phi(\lambda I)x}^{2}\leq \norm{\Omega}^{2} \norm{\Phi(\lambda I)}^{2}\norm{x}^{2}.
    \end{equation*}
    Combining the submultiplicativity with two inequalities above, we have 
    \begin{equation*}
        \norm{\Phi(B)x}^{2} \leq   b\norm{\Omega^{-1}}^{2}\max_{\lambda^{(1)}\leq \lambda \leq \lambda^{(b)}}\sum_{j=1}^{b}\varphi_{j}^{2}(\lambda)
        \leq b\norm{\Omega^{-1}}^{2}\norm{\Omega}^{2} \max_{\lambda^{(1)}\leq \lambda \leq \lambda^{(b)}}\norm{\Phi(\lambda I)}^{2}\norm{x}^{2},
    \end{equation*}
    which proves the lemma.
\end{proof}

\subsection{Solvents and latent roots}
Similar to scalar polynomials, we can also define the generalization of roots as solvents and latent roots~\cite[Sec.~1]{Dennis1976}.
\begin{definition}[Solvent and latent root/vector]
    Given a matrix polynomial $\Phi$ of the form \cref{eq:defPhiX}, if $\Phi(B)=0$, then $B$ is called a left solvent of $\Phi$. If $\Phi(\lambda I)$ is singular for some $\lambda\in\R$, then $\lambda$ is called a latent root of $\Phi$. If $\omega^{\Ttran}\Phi(\lambda I)=0$ for a nonzero vector $\omega$ and a latent root $\lambda$, then $\omega$ is called the corresponding left latent vector.
\end{definition}

In the rest of this paper, we will omit ``left'' before left solvent and left latent vector since we only consider matrix polynomials like \cref{eq:defPhiX}.
Let $\Phi\in\poly_{d}(\R^{b\times b})$ and $B$ be a solvent of $\Phi$. The generalized B\'ezout's theorem \cite[Chap~4.3, Thm.~1]{Gantmacher1998} asserts the existence of a matrix polynomial $\Phi_{B}\in\poly_{d-1}(\R^{b\times b})$ such that
\begin{equation*}
    \Phi(\lambda I) = (\lambda I-B)\Phi_{B}(\lambda I) \quad \forall\, \lambda\in\R.
\end{equation*}
The following lemma, essentially from \cite[Lem.~4.1]{Dennis1976}, provides a further connection between latent pairs and solvent.

\begin{lemma}
    \label{prop:Bezout}
    Suppose that $B$ admits a spectral decomposition $B=\Omega^{-1}\Lambda\Omega$ with real eigenvalues. 
    Given a matrix polynomial $\Phi\in\poly_{d}(\R^{b\times b})$, then $B$ is a solvent of $\Phi$ if and only $\omega^{\Ttran}\Phi(\lambda I) = 0$ holds for all left eigenpairs $(\lambda,\omega)$ of $B$.
\end{lemma}
\begin{proof}
    If $B$ is a solvent, then the result is directly from the generalized B\'ezout's theorem. In the other direction, suppose that $\Phi$ takes the form in \cref{eq:defPhiX}, then
    \begin{equation*}
        \Phi(B) = \Omega^{-1}\sum_{i=0}^{d}\Lambda^{i}\Omega C_{i}.
    \end{equation*}
    Note that for any left eigenpair $(\lambda,\omega)$ of $B$, it holds that
    \begin{equation*}
        0 = \omega^{\Ttran}\Phi(\lambda I) = \sum_{i=0}^{d}\lambda^{i}\omega^{\Ttran}C_{i}.
    \end{equation*}
    Stacking all left eigenpairs of $B$ together, we know that $\sum_{i=0}^{d}\Lambda^{i}\Omega C_{i}=0$, and in turn, $\Phi(B)=0$.
\end{proof}

\subsection{Interpolation and fundamental matrix polynomials}

The next ingredient we need from matrix polynomials is interpolation. Let 
\begin{equation}
    \label{eq:defvand}
    \van \defi 
    \begin{bmatrix}
        I & B_{1} & \cdots & B_{1}^{d-1}\\ 
        I & B_{2} & \cdots & B_{2}^{d-1}\\ 
        \vdots & \vdots & \ddots & \vdots \\ 
        I & B_{d} & \cdots & B_{d}^{d-1}\\ 
    \end{bmatrix}\in\R^{bd\times bd}
\end{equation}
be a block Vandermonde matrix, where $B_{i}\in\R^{b\times b}$. When $b=1$, that is, $\Phi$ is a scalar polynomial and $B_{i}\in\R$, a nonsingular (scalar) Vandermonde matrix $\van$ allows us to interpolate $\Phi$ with Lagrange polynomials $\Phi_{i}$  such that $\Phi_{i}(B_{j})=\delta_{ij}$, where $\delta_{ij}=1$ if $i=j$ and $0$ otherwise. When $b>1$, consider the following generalization.

\begin{definition}[Fundamental matrix polynomials]
    \label{def:fp}
    Consider $B_{1},\dotsc,B_{d}\in\R^{b\times b}$. Let $\{\fp_{1},\dotsc,\fp_{d}\}$ be a set of matrix polynomials that satisfy 
    \begin{equation*}
        \fp_{i}(B_{j}) =  \delta_{ij} I \quad\text{for all}\quad  1\leq i\neq j\leq d.
    \end{equation*}
    Then, we say $\{\fp_{1},\dotsc,\fp_{k}\}$ is a set of fundamental matrix polynomials. 
\end{definition}

With fundamental matrix polynomials defined above, \cite[Thm.~5.2 and Cor.~5.1]{Dennis1976}
provide a result for matrix polynomials analogous to Lagrange interpolation, as follows.
\begin{proposition}
    \label{prop:interpolation}
    Given matrices $B_{1},\dotsc,B_{d}\in\R^{b\times b}$, if the matrix $\van$ defined in~\cref{eq:defvand} is nonsingular, then there exists a unique set of fundamental matrix polynomials $\{\fp_{1},\dotsc,\fp_{d}\}$. Moreover, for any matrix polynomial $\Phi\in\poly_{d-1}(\R^{b\times b})$, it holds that  
    \begin{equation*}
        \Phi(\lambda I) = \sum_{k=1}^{d} \fp_{k}(\lambda I)\Phi(B_{k})\quad \forall\, \lambda\in\R.
    \end{equation*}
\end{proposition}

We should note that, different from the scalar case, $B_{1},\dotsc,B_{d}$ having disjoint spectra is not sufficient to guarantee the nonsingularity of $\van$. For example, let $B_{1}=\begin{bmatrix}
    1 & \\ & 2
\end{bmatrix}$ and $B_{2}=\begin{bmatrix}
    2 & 1\\ -1 & 1
\end{bmatrix}$. Then $B_{1}$ and $B_{2}$ have disjoint spectra, but $\van\cdot[1,-2,-1,1]^{\Ttran}=0$. However, the following proposition shows that $\van$ is nonsingular generically. In particular, if $[\Omega_{1},\dotsc,\Omega_{d}]$ is a Gaussian random matrix, then $\van$ is nonsingular almost surely.

\begin{proposition}
    \label{prop:nonsingular}
    Let $\Omega_{1},\dotsc,\Omega_{d}\in\R^{b\times b}$, and $B_{i}=\Omega_{i}^{-1}\Lambda_{i}\Omega_{i}$ for $i=1,\dotsc,d$ with real diagonal matrices $\Lambda_{i}$. Suppose that $\Lambda_{i}$ and $\Lambda_{j}$ have disjoint spectra for $i\neq j$, then $\van$ is nonsingular generically with respect to $\Omega_{1},\dotsc,\Omega_{d}$ in Lebesgue measure.
\end{proposition}
\begin{proof}
    Let
    \begin{equation*}
        f(\Omega_{1},\dotsc,\Omega_{d})\defi \det \begin{bmatrix}
        \Omega_{1} & \Lambda_{1}\Omega_{1} & \cdots & \Lambda_{1}^{d-1}\Omega_{1}\\ 
        \Omega_{2} & \Lambda_{2}\Omega_{2} & \cdots & \Lambda_{2}^{d-1}\Omega_{2}\\  
        \vdots & \vdots & \ddots & \vdots \\
        \Omega_{d} & \Lambda_{d}\Omega_{d} & \cdots & \Lambda_{d}^{d-1}\Omega_{d}\\ 
    \end{bmatrix}.
    \end{equation*}
    It is clear that $f$ is a polynomial of entries of $\Omega_{i}$. Since $\Lambda_{i}$ and $\Lambda_{j}$ have disjoint spectra for $i\neq j$, we know that $f(I_{b},\dotsc,I_{b})\neq 0$, yielding that $f$ is a nonzero polynomial. Note that 
    \begin{equation*}
        \det\van = \det \begin{bmatrix}
            \Omega_{1}&&\\ 
            &\ddots &\\ 
            &&\Omega_{d}
        \end{bmatrix}^{-1}
        \det \begin{bmatrix}
        \Omega_{1} & \Lambda_{1}\Omega_{1} & \cdots & \Lambda_{1}^{d-1}\Omega_{1}\\ 
        \Omega_{2} & \Lambda_{2}\Omega_{2} & \cdots & \Lambda_{2}^{d-1}\Omega_{2}\\  
        \vdots & \vdots & \ddots & \vdots \\
        \Omega_{d} & \Lambda_{d}\Omega_{d} & \cdots & \Lambda_{d}^{d-1}\Omega_{d}\\ 
    \end{bmatrix}
    =\Bigl(\prod_{i=1}^{d}\det(\Omega_{i})\Bigr)^{-1}f(\Omega_{1},\dotsc,\Omega_{d}).   
    \end{equation*}
    We know that $\det \van\neq 0$ generically.
\end{proof}

\subsection{Explicit formula for fundamental matrix polynomials}

The fundamental theorem of algebra asserts that a scalar polynomial can be factored into a product of (complex) linear functions. This property, however, does not generally extend to matrix polynomials. Nevertheless, under certain additional assumptions on the solvents, the following lemma shows that a matrix polynomial $\Phi$ can indeed be decomposed into a product of some degree-one matrix polynomials when the variable is a scalar matrix.

\begin{lemma}
    \label{lem:Chain}
    Suppose that $B_{1},\dotsc,B_{d}\in\R^{b\times b}$ admit spectral decompositions $B_{i}=\Omega_{i}^{-1}\Lambda_{i}\Omega_{i}$, where $\Lambda_{i}$ are real diagonal matrices for $i=1,\dotsc,d$. Assume $\Lambda_{i}$ and $\Lambda_{j}$ have disjoint spectra for $i\neq j$.
    We define the following quantities recursively: For each $i=d,d-1,\dotsc,2,1$, let  
    \begin{equation*}
        \widehat{\Omega}_{i} \defi \Omega_{i}S_{i,d}\quad\text{and}\quad  
        \widehat{B}_{i}\defi \widehat{\Omega}_{i}^{-1}\Lambda_{i}\widehat{\Omega}_{i},
    \end{equation*}
    where $S_{i,d}\in\R^{b\times b}$ is defined by the following recurrence:
    \begin{equation*}
        S_{i,i} \defi  I_{b},\quad  S_{i,j} \defi  B_{i}S_{i,j-1}-S_{i,j-1}\widehat{B}_{j} \quad\text{for}\quad j=i+1,\dotsc,d.
    \end{equation*} 
    Assuming that the matrix $\van$ in \cref{eq:defvand} is nonsingular, and all $S_{i,d}$ are nonsingular for $1\leq i\leq d$, then the fundamental matrix polynomial $\fp_{1}$ defined in \cref{def:fp} admits the following formula: 
    \begin{equation}
    \label{eq:deffp1}
        \fp_{1}(\lambda I) = \Bigl(\prod_{i=2}^{d}(\lambda I-\widehat{B}_{i})\Bigr)S_{1,d}^{-1}\quad \forall\, \lambda\in\R.
    \end{equation}
\end{lemma}

When $b=1$, that is, $B_{i}=\widetilde{B}_{i}=\Lambda_{i}$ is a scalar, \cref{lem:Chain} yields that  
\begin{equation*}
    S_{1,d} = (\lambda_{1}-\lambda_{d-1})S_{1,d-1} = \dotsb = \prod_{j=2}^{d}(\lambda_{1}-\lambda_{d}).
\end{equation*}
In turn, $F_{1}(\cdot)$ reduces to the standard Lagrange basis of polynomials:
\begin{equation*}
    F_{1}(\lambda) = \prod_{i=2}^{d}\frac{\lambda-\lambda_{i}}{\lambda_{1}-\lambda_{i}}.
\end{equation*}
Before proving \cref{lem:Chain}, we provide a diagram to illustrate the recursive construction of variables in \cref{lem:Chain} as \cref{fig:diagBinChain}. 
The construction involves a two-level loop. In the outer loop, we define $\widehat{\Omega}_{i}$ and $\widehat{B}_{i}$ in a backward manner: Starting with $\widehat{\Omega}_{d}=\Omega_{d}$ and $\widehat{B}_{d}=B_{d}$, we recursively define $\widehat{\Omega}_{d-1},\widehat{B}_{d-1}, \widehat{\Omega}_{d-2},\widehat{B}_{d-2}$, and so on. 
To define $\widehat{\Omega}_{i}$ and $\widehat{B}_{i}$, assuming that $\widehat{\Omega}_{i+1},\dotsc,\widehat{\Omega}_{d}$ and $\widehat{B}_{i+1},\dotsc,\widehat{B}_{d}$ are already constructed, we enter the inner loop, where we define $S_{i,i+1},\dotsc,S_{i,d}$ in a forward manner. These intermediate matrices are then used to define $\widehat{\Omega}_{i}=\Omega_{i}S_{i,d}$ and subsequently $\widehat{B}_{i}=\widehat{\Omega}_{i}^{-1}\Lambda_{i}\widehat{\Omega}_{i}$.

\begin{figure}[H]
    \centering
\begin{tikzpicture}[
  node distance=0.6cm and 1.5 cm,
  every node/.style={draw, minimum width=1.2cm, minimum height=1cm, align=center}
]
\node (a) {$B_{i},\Omega_{i}$};
\node[right=of a] (b) {$B_{i+1},\Omega_{i+1}$};
\node[right=1.5cm of b] (c) {$B_{d-1},\Omega_{d-1}$};
\node[right=of c] (d) {$B_{d},\Omega_{d}$};

\node[below=of a] (a2) {$\widehat{B}_{i},\widehat{\Omega}_{i}$};
\node[below=of b] (b2) {$\widehat{B}_{i+1},\widehat{\Omega}_{i+1}$};
\node[below=of c] (c2) {$\widehat{B}_{d-1},\widehat{\Omega}_{d-1}$};
\node[below=of d] (d2) {$\widehat{B}_{d},\widehat{\Omega}_{d}$};

\draw[->, dashed, shorten <=2pt, >=Latex, line width=0.8pt] (c2) -- (b2);
\draw[->, shorten <=2pt, >=Latex, line width=0.8pt] (d2) -- (c2);
\draw[->, shorten <=2pt, >=Latex, line width=0.8pt] (b2) -- (a2);

\draw[->, shorten <=2pt, >=Latex, line width=0.8pt] (d) -- (d2);
\draw[->, shorten <=2pt, >=Latex, line width=0.8pt] (c) -- (c2);
\draw[->, shorten <=2pt, >=Latex, line width=0.8pt] (b) -- (b2);
\draw[->, shorten <=2pt, >=Latex, line width=0.8pt] (a) -- (a2);

\node[below=of c2] (c3) {$S_{d-1,d-1}$};
\node[below=of d2] (d3) {$S_{d-1,d}$};
\node[draw=none, fill=none, inner sep=0pt,below=of a2] (a3) {};
\node[draw=none, fill=none, inner sep=0pt,below=of b2] (b3) {};

\draw[->, shorten <=2pt, >=Latex, line width=0.8pt] (d2) -- (d3);
\draw[->, shorten <=2pt, >=Latex, line width=0.8pt] (c3) -- (d3);

\draw[->, shorten <=2pt, >=Latex, line width=0.8pt] (d3) -- (c2);

\node[draw=none, fill=none, inner sep=0pt,below=0.1cm of a3] (a4) {};
\node[draw=none, fill=none, inner sep=0pt,below=0.1cm of b3] (b4) {$\vdots$};

\node[draw=none, fill=none, inner sep=0pt,right=0.6cm of b4]{$\vdots$};
\node[draw=none, fill=none, inner sep=0pt,below=0.1cm of c3] (c4) {$\vdots$};
\node[draw=none, fill=none, inner sep=0pt,below=0.1cm of d3] (d4) {$\vdots$};


\node[below=0.1cm of a4] (a5) {$S_{i,i}$};
\node[below=0.1cm of b4] (b5) {$S_{i,i+1}$};

\node[below=0.1cm of c4] (c5) {$S_{i,d-1}$};
\node[below=0.1cm of d4] (d5) {$S_{i,d}$};

\draw[->, shorten <=2pt, >=Latex, line width=0.8pt] (a5) -- (b5);

\draw[->, dashed, shorten <=2pt, >=Latex, line width=0.8pt] (b5) -- (c5);
\draw[->, shorten <=2pt, >=Latex, line width=0.8pt] (c5) -- (d5);

\draw[->, dashed, shorten <=2pt, >=Latex, line width=0.8pt, bend right=40] ($ (b2)!0.5!(c2) $) to ($ (b5)!0.5!(c5) $) ;
\draw[->, shorten <=2pt, >=Latex, line width=0.8pt, bend right=40] (b2) to (b5);
\draw[->, shorten <=2pt, >=Latex, line width=0.8pt, bend right=40] (c2) to (c5);
\draw[->, shorten <=2pt, >=Latex, line width=0.8pt] (d5) to (a2);
\draw[->, shorten <=2pt, >=Latex, line width=0.8pt, bend left=40] (d2) to (d5);

\end{tikzpicture}

    \caption{Diagram for defining $\widehat{B}_{i}$ and $\widehat{\Omega}_{i}$.}
    \label{fig:diagBinChain}
\end{figure}

\begin{proof}[Proof of \cref{lem:Chain}]
    The proof is based on mathematical induction. Denote $\Phi_{d-1}=\fp_{1}\in\poly_{d-1}(\R^{b\times b})$.
    For the base case, since $\Phi_{d-1}(B_{d})=0$, the generalized B\'ezout's theorem yields the existence of a matrix polynomial $\Phi_{d-2}\in\poly_{d-2}(\R^{b\times b})$, such that  
    \begin{equation*}
        \Phi_{d-1}(\lambda I) = (\lambda I-B_{d})\Phi_{d-2}(\lambda I) = (\lambda I-\widehat{B}_{d})\Phi_{d-2}(\lambda I) \quad \forall\,\lambda\in\R.
    \end{equation*}
    We then consider the general case.
    Suppose that we have 
    \begin{equation}
        \label{eq:indasp}
        \Phi_{d-1}(\lambda I) = \Bigl(\prod_{j=i+1}^{d}(\lambda I-\widehat{B}_{j})\Bigr)\Phi_{i-1}(\lambda I) \quad \forall\,\lambda\in\R
    \end{equation}
    for some $2\leq i\leq d-1$, where $\Phi_{i-1}\in\poly_{i-1}(\R^{b\times b})$.
    Since $B_{i}$ is a solvent of $\Phi_{d-1}$, for any left eigenpair $(\lambda_{i},\omega_{i})$ of $B_{i}$, \cref{prop:Bezout} yields that 
    \begin{equation}
        \label{eq:lemLamMat1}
        0 = \omega_{i}^{\Ttran}\Phi_{d-1}(\lambda_{i} I) = \omega_{i}^{\Ttran}\Bigl(\prod_{j=i+1}^{d}(\lambda_{i} I-\widehat{B}_{j})\Bigr)\Phi_{i-1}(\lambda_{i} I).
    \end{equation}
    Since $\Lambda_{i}$ and $\Lambda_{j}$ have disjoint spectra, we know that the vector  $\omega_{i}^{\Ttran}\prod_{j=i+1}^{d}(\lambda_{i} I-\widehat{B}_{j})$ is nonzero, and therefore is a latent vector of $\Phi_{i-1}$ corresponding to the latent root $\lambda_{i}$.
    Since
    \begin{equation}
        \label{eq:solventS}
        \begin{aligned}
            &\omega_{i}^{\Ttran}\prod_{j=i+1}^{d}(\lambda_{i} I-\widehat{B}_{j}) = (\omega_{i}^{\Ttran}B_{i}-\omega_{i}^{\Ttran}\widehat{B}_{i+1})\prod_{j=i+2}^{d}(\lambda_{i} I-\widehat{B}_{j})\\ 
            =\ &\omega_{i}^{\Ttran}S_{i,i+1}\prod_{j=i+2}^{d}(\lambda_{i} I-\widehat{B}_{j}) = (\omega_{i}^{\Ttran}B_{i}S_{i,i+1}-\omega_{i}^{\Ttran}S_{i,i+1}\widehat{B}_{i+2})\prod_{j=i+3}^{d}(\lambda_{i} I-\widehat{B}_{j})\\ 
            =\ &\omega_{i}^{\Ttran}S_{i,i+2}\prod_{j=i+3}^{d}(\lambda_{i} I-\widehat{B}_{j})= \dotsb = \omega_{i}^{\Ttran}S_{i,d-1}(\lambda_{i} I-\widehat{B}_{d}) = \omega_{i}^{\Ttran}S_{i,d},
        \end{aligned}
    \end{equation}
    we know that $\widehat{\Omega}_{i}=\Omega_{i}S_{i,d}$ is a nonsingular matrix whose rows are the latent vectors of $\Phi_{i-1}$, with the corresponding latent roots populating the diagonal of $\Lambda_{i}$.
    Then \cref{prop:Bezout} yields that $\widehat{B}_{i}$ is a solvent of $\Phi_{i-1}$. In turn, the generalized B\'ezout's theorem asserts the existence of a matrix polynomial $\Phi_{i-2}\in\poly_{i-2}(\R^{b\times b})$, such that 
    \begin{equation*}
        \Phi_{i-1}(\lambda I) = (\lambda I-\widehat{B}_{i})\Phi_{i-2}(\lambda I)\quad \forall\,\lambda\in\R.
    \end{equation*}
    Plugging it in \cref{eq:indasp} and using mathematical induction, we obtain 
    \begin{equation*}
        \Phi_{d-1}(\lambda I) = \Bigl(\prod_{i=2}^{d}(\lambda I-\widehat{B}_{i})\Bigr)\Phi_{0}\quad \forall\,\lambda\in\R,
    \end{equation*}
    where $\Phi_{0}$ is a coefficient matrix to be determined. Note that $\Phi_{d-1}(B_{1}) = \fp_{1}(B_{1})=I$, that is, $B_{1}$ is a solvent of $\Phi_{d-1}(\cdot)-I$. With arguments similar to \cref{eq:lemLamMat1,eq:solventS}, we know $\Omega_{1}S_{1,d}\Phi_{0} - \Omega_{1} = 0$, implying that $\Phi_{0}=S_{1,d}^{-1}$.
\end{proof}

The matrices $\widehat{B}_{i}$ from \cref{lem:Chain} for $i=2,\dotsc,d$ are referred to as a chain of solvents \cite[Def.~4.1]{Dennis1976} of the matrix polynomial $\fp_{1}$. This chain of  solvents can be used to represent the coefficients $C_{i}$ in \cref{eq:defPhiX} via elementary symmetric (matrix-valued) functions \cite[Thm.~4.6]{Dennis1976}.
It is important to note that, in general, \emph{only} $\widehat{B}_{d}=B_{d}$  is guaranteed to be a solvent of $\fp_{1}$.
Another remark regarding \cref{lem:Chain} is that the nonsingularity of $\van$ is \emph{not} a sufficient condition for ensuring the nonsingularity of all $S_{i,d}$. This is due to the fact that certain submatrices of $\van$ with block Vandermonde structure, such as $\begin{bmatrix}
    I & B_{1}\\ I & B_{2}
\end{bmatrix}$, can be singular even if $\van$ itself is nonsingular, see \cite[p.~841]{Dennis1976} for details. 
In the following, we show that the fundamental matrix polynomial $\fp_{1}$ in~\cref{eq:deffp1} is well-defined generically.
\begin{proposition}
    \label{prop:wellpose}
    Let $\Omega_{1},\dotsc,\Omega_{d}\in\R^{b\times b}$, and $B_{i}=\Omega_{i}^{-1}\Lambda_{i}\Omega_{i}$ for $i=1,\dotsc,d$ with real diagonal matrices $\Lambda_{i}$. Suppose that $\Lambda_{i}$ and $\Lambda_{j}$ have disjoint spectra for $i\neq j$, then all $S_{i,d}$ in \cref{lem:Chain} are nonsingular for $1\leq i\leq d$, and, hence, $\fp_{1}$ in~\cref{eq:deffp1} is well-defined generically with respect to $\Omega_{1},\dotsc,\Omega_{d}$ in Lebesgue measure. 
\end{proposition}
\begin{proof}
    The proof follows a similar idea to the proof of \cref{prop:nonsingular}.
    We first show that if all $\Omega_{i}$ are identity matrices, then 
    \begin{equation}
        \label{eq:appind}
        S_{i,k} \text{ is a nonsingular diagonal matrix}\quad\text{and}\quad \widehat{B}_{i}=\Lambda_{i} 
    \end{equation}
    holds for all $1\leq i\leq k\leq d$.
    We verify this by an induction for $i$ from $d$ to $1$. The basic situation comes from the definition $S_{d,d}=I$ and $\widehat{B}_{d}=\Omega_{d}^{-1}\Lambda_{d}\Omega_{d}=\Lambda_{d}$. Suppose that \cref{eq:appind} holds for $i+1,\dotsc,d$.
    For any $i< k\leq d$, since 
    \begin{equation}
        \label{eq:relationSik}
        \begin{aligned}
            S_{i,k} &= B_{i}S_{i,k-1}-S_{i,k-1}\widehat{B}_{k} = (B_{i}-S_{i,k-1}\widehat{B}_{k}S_{i,k-1}^{-1})S_{i,k-1}= \dotsb \\ 
            &= \Bigl(\prod_{j=1}^{k-i}(B_{i}-S_{i,k-j}\widehat{B}_{k-j+1}S_{i,k-j}^{-1})\Bigr)S_{i,i} 
            = \prod_{j=1}^{k-i}(B_{i}-S_{i,k-j}\widehat{B}_{k-j+1}S_{i,k-j}^{-1}) 
            \\ 
            & = \prod_{j=1}^{k-i}\bigl(\Omega_{i}^{-1}\Lambda_{i}\Omega_{i}-(\widehat{\Omega}_{k-j+1}S_{i,k-j}^{-1})^{-1}\Lambda_{k-j+1}(\widehat{\Omega}_{k-j+1}S_{i,k-j}^{-1})\bigr)\\ 
            & = \prod_{j=1}^{k-i}\bigl(\Lambda_{i}-(S_{k-j+1,d}S_{i,k-j}^{-1})^{-1}\Lambda_{k-j+1}(S_{k-j+1,d}S_{i,k-j}^{-1})\bigr),
        \end{aligned}
    \end{equation}
    where we use $\Omega_{i}=I_{d}$ in the last equality. 
    Making another induction for $k$ from $i$ to $d$ and recalling that $\widehat{B}_{i}\defi \widehat{\Omega}_{i}^{-1}\Lambda_{i}\widehat{\Omega}_{i}$, we obtain \cref{eq:appind}.
    Now taking determinant on both sides of \cref{eq:relationSik}, we have 
    \begin{equation*}
        \begin{aligned}
            \det S_{i,k} &= \prod_{j=1}^{k-i}\det \bigl(\Omega_{i}^{-1}\Lambda_{i}\Omega_{i}-(\widehat{\Omega}_{k-j+1}S_{i,k-j}^{-1})^{-1}\Lambda_{k-j+1}(\widehat{\Omega}_{k-j+1}S_{i,k-j}^{-1})\bigr)\\ 
            &= \prod_{j=1}^{k-i}\det \begin{bmatrix}
                I & I \\ 
                (\widehat{\Omega}_{k-j+1}S_{i,k-j}^{-1})^{-1}\Lambda_{k-j+1}(\widehat{\Omega}_{k-j+1}S_{i,k-j}^{-1})
                & \Omega_{i}^{-1}\Lambda_{i}\Omega_{i}
            \end{bmatrix}.
        \end{aligned}
    \end{equation*}
    Note that $\Lambda_{k-j+1}$ and $\Lambda$ have disjoint spectra, we can perform the same trick in \cref{prop:nonsingular} to show that each factor in $\det S_{i,k}$ is nonzero generically, and in turn, $S_{i,d}$ is nonsingular generically.
\end{proof}

We conclude this section by providing expressions for $\fp_{k}$ for a generic $1 \leq k \leq d$, in an analogous manner. 
\begin{theorem}
    \label{thm:Chain}
    Suppose that $B_{1},\dotsc,B_{d}\in\R^{b\times b}$ admit spectral decompositions $B_{i}=\Omega_{i}^{-1}\Lambda_{i}\Omega_{i}$, where $\Lambda_{i}$ are real diagonal matrices for $i=1,\dotsc,d$. Assume $\Lambda_{i}$ and $\Lambda_{j}$ have disjoint spectra for $i\neq j$. For any $1\leq k\leq d$, denote
    \begin{equation*}
        \Omega_{i}^{(k)},B_{i}^{(k)},\Lambda_{i}^{(k)} \defi 
        \begin{cases}
            \Omega_{k},B_{k},\Lambda_{k} & \text{if } i=1, \\
            \Omega_{i-1},B_{i-1},\Lambda_{i-1} & \text{if } 2\leq i\leq k, \\
            \Omega_{i},B_{i},\Lambda_{i}  & \text{if } k+1\leq i\leq d,
        \end{cases} 
    \end{equation*}  
    We define the following quantities recursively: For each $i=d,d-1,\dotsc,2,1$, let  
    \begin{equation*}
        \widehat{\Omega}_{i}^{(k)} \defi \Omega_{i}^{(k)}S_{i,d}^{(k)},\quad 
        \widehat{B}_{i}^{(k)}\defi \bigl(\widehat{\Omega}_{i}^{(k)}\bigr)^{-1}\Lambda_{i}^{(k)}\widehat{\Omega}_{i}^{(k)},
    \end{equation*}
    where $S_{i,d}^{(k)}\in\R^{b\times b}$ is defined by the following recurrence:
    \begin{equation*}
        S_{i,i}^{(k)} \defi  I_{b},\quad  S_{i,j}^{(k)} \defi  B_{i}^{(k)}S_{i,j-1}^{(k)}-S_{i,j-1}^{(k)}\widehat{B}_{j}^{(k)} \quad\text{for}\quad j=i+1,\dotsc,d.
    \end{equation*} 
    Assuming that the matrix $\van$ in \cref{eq:defvand} is nonsingular, and all $S_{i,d}^{(k)}$ are nonsingular for $1\leq i\leq d$, then the fundamental matrix polynomials $\fp_{k}$ defined in \cref{def:fp} admits the following formula: 
    \begin{equation}
    \label{eq:deffp}
        \fp_{k}(\lambda I) = \Bigl(\prod_{i=2}^{d}(\lambda I-\widehat{B}_{i}^{(k)})\Bigr)\bigl(S_{1,d}^{(k)}\bigr)^{-1}\quad \forall\, \lambda\in\R.
    \end{equation}
\end{theorem}

We remark that \cref{eq:deffp1} is a special case of \cref{eq:deffp}. Moreover, similar to \cref{prop:wellpose}, the expression for $\fp_{k}$ in \cref{eq:deffp} is also well-defined generically.

\subsection{Growth of fundamental matrix polynomials}

Suppose that all eigenvalues of $B_{i}$ for $1\leq i\leq d$ are contained in $[\lambda_{\min}^{\cl},\lambda_{\max}^{\cl}]$, where $\cdot^{\cl}$ means cluster.
The next step is to control the growth of $\norm{\fp_{k}(\lambda I)}$ as $\lambda$ moves away from $[\lambda_{\min}^{\cl},\lambda_{\max}^{\cl}]$.
To this end, we introduce some quantities that allow us to derive bounds analogous to those for the Lebesgue function in the scalar polynomial case.

To avoid cumbersome notation involving the superscript $\cdot^{(k)}$, we first define the relevant quantities for $\fp_{1}$.
Recalling that $\widehat{B}_{i}=\widehat{\Omega}_{i}^{-1}\Lambda_{i}\widehat{\Omega}_{i}$, we let  
\begin{equation*}
    \chim^{(1)} \defi  \max_{2\leq i\leq d}\Bigl\{ \frac{\norm{\lambda_{\min}^{\cl}I-\widehat{B}_{i}}}{\norm{\lambda_{\min}^{\cl}I-\Lambda_{i}}},\frac{\norm{\lambda_{\max}^{\cl}I-\widehat{B}_{i}}}{\norm{\lambda_{\max}^{\cl}I-\Lambda_{i}}},1\Bigr\}.
\end{equation*}
For any $\lambda<\lambda_{\min}^{\cl}$, by \cref{eq:deffp} and submultiplicativity,
\begin{equation*}
    \begin{aligned}
    \norm{\fp_{1}(\lambda I)}&\leq \Bigl(\prod_{i=2}^{d}\norm{\lambda I-\widehat{B}_{i}}\Bigr) \norm{S_{1,d}^{-1}} \leq \Bigl(\prod_{i=2}^{d}\bigl(\lambda_{\min}^{\cl}-\lambda+\norm{\lambda_{\min}^{\cl} I-\widehat{B}_{i}}\bigr)\Bigr) \norm{S_{1,d}^{-1}}
    \\ &\leq \Bigl(\prod_{i=2}^{d}\bigl(\lambda_{\min}^{\cl}-\lambda+\chim^{(1)}\norm{\lambda_{\min}^{\cl} I-\Lambda_{i}}\bigr)\Bigr) \norm{S_{1,d}^{-1}}\\  
    &\leq \bigl((\lambda_{\max}^{\cl}-\lambda)\chim^{(1)}\bigr)^{d-1}\norm{S_{1,d}^{-1}}.
    \end{aligned}
\end{equation*}
Similarly, when $\lambda>\lambda_{\max}^{\cl}$, it holds that 
\begin{equation*}
    \norm{\fp_{1}(\lambda I)}\leq \bigl((\lambda-\lambda_{\min}^{\cl})\chim^{(1)}\bigr)^{d-1}\norm{S_{1,d}^{-1}}.
\end{equation*}
Now we need to quantify the coefficient term $\norm{S_{1,d}^{-1}}$. Taking $i=1$ and $k=d$ in \cref{eq:relationSik}, we have  
\begin{equation*}
    S_{1,d} = \prod_{i=1}^{d-1}\bigl(\Omega_{1}^{-1}\Lambda_{1}\Omega_{1}-(\widehat{\Omega}_{d-i+1}S_{1,d-i}^{-1})^{-1}\Lambda_{d-i+1}(\widehat{\Omega}_{d-i+1}S_{1,d-i}^{-1})\bigr).
\end{equation*}
Since each factor in the product is the difference between two matrices with eigenvalues $\Lambda_{1}$ and $\Lambda_{d-i+i}$, respectively, a proper scaling-independent quantity for measuring  $\norm{S_{1,d}^{-1}}$ is 
\begin{equation*}
    \chic^{(1)}\defi \norm{S_{1,d}^{-1}}^{\frac{1}{d-1}}\min_{\substack{2\leq i\leq d\\ \lambda_{1}\in \Lambda_{1}, \lambda_{i}\in \Lambda_{i}}} \abs{\lambda_{1}-\lambda_{i}}.
\end{equation*}
For a generic $1\leq k\leq d$, with notation in \cref{thm:Chain}, we define 
\begin{equation*}
    \begin{aligned}
        \chim^{(k)} &\defi  \max_{2\leq i\leq d}\Bigl\{ \frac{\bignorm{\lambda_{\min}^{\cl}I-\widehat{B}_{i}^{(k)}}}{\bignorm{\lambda_{\min}^{\cl}I-\Lambda_{i}^{(k)}}},\frac{\bignorm{\lambda_{\max}^{\cl}I-\widehat{B}_{i}^{(k)}}}{\bignorm{\lambda_{\max}^{\cl}I-\Lambda_{i}^{(k)}}},1\Bigr\},\\ 
        \chic^{(k)} &\defi \bignorm{(S_{1,d}^{(k)})^{-1}}^{\frac{1}{d-1}}\min_{\substack{2\leq i\leq d\\ \lambda_{1}^{(k)}\in \Lambda_{1}^{(k)}, \lambda_{i}^{(k)}\in \Lambda_{i}^{(k)}}} \abs{\lambda_{1}^{(k)}-\lambda_{i}^{(k)}},
    \end{aligned}
\end{equation*}
and their uniform bound as 
\begin{equation}
    \label{eq:defchi}
    \chim \defi \max_{1\leq k\leq d}\chim^{(k)}
    \quad\text{and}\quad
    \chic \defi \max_{1\leq k\leq d}\chic^{(k)}.  
\end{equation}
We remark that in the scalar case ($b=1$), where all scalar products commute, or $\Lambda_{i}=\lambda_{i} I_{b}$ are scalar matrices for all $i$, the quantities $\chim$ and $\chic$ simplify significantly to $\chim=1$ and $\chic\leq 1$.

By applying \cref{thm:Chain} to \cref{prop:interpolation}, we obtain the main theorem of this section, which provides an explicit representation and a growth estimate for the interpolation of matrix polynomials.
\begin{theorem}
    \label{thm:interpolation}
    Let $\Phi\in\poly_{d-1}(\R^{b\times b})$ be a matrix polynomial. Suppose that matrices $B_{1},\dotsc,B_{d}\in\R^{b\times b}$ have disjoint real spectra contained in $[\lambda_{\min}^{\cl},\lambda_{\max}^{\cl}]$. Assuming that $\van$ defined in \cref{eq:defvand} is nonsingular, we have   
    \begin{equation*}
        \Phi(\lambda I) = \sum_{k=1}^{d} \fp_{k}(\lambda I)\Phi(B_{k}) \quad \forall\, \lambda\in\R.
    \end{equation*} 
    Additionally, assume that the expression for $\fp_{k}$ in \cref{eq:deffp} are well-defined for $1\leq k\leq d$. Then for any $\lambda \notin [\lambda_{\min}^{\cl},\lambda_{\max}^{\cl}]$, it holds that 
    \begin{equation*}
        \norm{\Phi(\lambda I)} \leq \sqrt{d}\cdot
        \Bignorm{
        \begin{bmatrix}
         \Phi(B_{1})\\ 
         \vdots\\ 
            \Phi(B_{d})
        \end{bmatrix}}
        \cdot \max_{1\leq k\leq d}\norm{\fp_{k}(\lambda I)}
    \end{equation*}
    and
    \begin{equation}
        \label{eq:Gd}
        \max_{1\leq k\leq d} \norm{\fp_{k}(\lambda I)}^{\frac{1}{d-1}} \leq 
        \frac{\max\{\lambda_{\max}^{\cl}-\lambda,\lambda-\lambda_{\min}^{\cl}\}}{\min\limits_{\substack{1\leq i\neq j\leq d\\ \lambda_{i}\in \Lambda_{i}, \lambda_{j}\in \Lambda_{j}}} \abs{\lambda_{i}-\lambda_{j}}} \cdot \chim \chic.
    \end{equation}
\end{theorem}

Consider $b\cdot d$ points $\lambda_{i_{1}}^{(i_{2})}$ in $[\lambda_{\min}^{\cl},\lambda_{\max}^{\cl}]$, where $1\leq i_{1}\leq d$ and $1\leq i_{2}\leq b$.
Compared with interpolation via scalar-valued Lagrange polynomials, using matrix polynomials in $\poly_{d-1}(\R^{b\times b})$ improves the growth of fundamental (scalar/matrix) polynomials from 
\begin{equation}
    \label{eq:cmpGrowth}
    \Bigl(\frac{\max\{\lambda_{\max}^{\cl}-\lambda,\lambda-\lambda_{\min}^{\cl}\}}{\min\limits_{\substack{ \abs{i_{1}-i_{2}}+\abs{j_{1}-j_{2}}\neq 0}} \abs{\lambda_{i_{1}}^{(i_{2})}-\lambda_{j_{1}}^{(j_{2})}}}\Bigr)^{bd-1} 
    \quad\text{to}\quad
    \Bigl(\frac{\max\{\lambda_{\max}^{\cl}-\lambda,\lambda-\lambda_{\min}^{\cl}\}}{\min\limits_{\substack{1\leq i\neq j\leq d\\ \lambda_{i}\in \Lambda_{i}, \lambda_{j}\in \Lambda_{j}}} \abs{\lambda_{i}-\lambda_{j}}} \cdot \chim \chic\Bigr)^{d-1}.
\end{equation}
\emph{If $\chim \chic$ is mild}, this improvement not only reduces the exponent from $bd-1$ to $b-1$ but also replaces the standard (relative) eigenvalue gap with the $b$th order (relative) eigenvalue gap. 
In particular, if $\lambda_{i_{1}}^{(i_{2})}=\lambda_{i_{1}}$ for all $1\leq i_{1}\leq d$ and $1\leq i_{2}\leq b$, that is, $\Lambda_{i}=\lambda_{i} I_{b}$ for all $1\leq i\leq d$, then the left-hand term in \cref{eq:cmpGrowth} becomes infinite as the denominator vanishes, whereas the right-hand term remains bounded by
\begin{equation*}
    \Bigl(\frac{\max\{\lambda_{\max}^{\cl}-\lambda,\lambda-\lambda_{\min}^{\cl}\}}{\min\limits_{1\leq i\neq j\leq d} \abs{\lambda_{i}-\lambda_{j}}} \Bigr)^{d-1}.
\end{equation*}

\section{Cluster robustness of RSBL}

In this section, we present the main result of this paper along with empirical evidence supporting our conjecture. We also mention an intrinsic challenge in the analysis of RSBL.

\subsection{Main theoretical result}
The following theorem provides a structural bound for the cluster robustness of RSBL defined in \cref{eq:defangle}.
\begin{theorem}
    \label{thm:sb}
    Let $A\in\R^{n\times n}$ be a symmetric matrix with spectral decomposition
    \begin{equation}
        \label{eq:decompA}
        A = \begin{bmatrix}
            Q & Q_{\perp}
        \end{bmatrix} 
        \begin{bmatrix}
            \Lambda&\\ 
            &\Lambda_{\perp}
        \end{bmatrix}
        \begin{bmatrix}
            Q & Q_{\perp}
        \end{bmatrix}^{\Ttran},
        \quad \text{where}\quad \Lambda = \diagm\{ \Lambda_{1},\dotsc,\Lambda_{d} \}
    \end{equation}
    with $\Lambda_{i}\in\R^{b\times b}$ for $1\leq i\leq d$.
    Assume that all eigenvalues in $\Lambda$ are contained in an interval $[\lambda_{\min}^{\cl},\lambda_{\max}^{\cl}]$ and eigenvalues in $\Lambda_{\perp}$ are outside $[\lambda_{\min}^{\cl},\lambda_{\max}^{\cl}]$. 
    Assume $m=n/b$ is an integer.
    Let $\Omega\in\R^{n\times b}$ be a Gaussian random initial matrix with partition 
    \begin{equation*}
        [
            Q, Q_{\perp}
        ]^{\Ttran}\Omega
        = 
        \begin{bmatrix}
            \Omega_{1}\\ \vdots \\ \Omega_{m}
        \end{bmatrix},
    \end{equation*}
    where $\Omega_{i}\in\R^{b\times b}$ for $1\leq i\leq m$. Assume that the $b$th order eigenvalue gap $\gap_{b}$ in \cref{eq:defgap} is positive. Let  
\begin{equation*}
    \begin{aligned}
        c_{\Omega} :&=  \sqrt{dn-bd^{2}}\cdot \max\limits_{1\leq i\leq d}\norm{\Omega_{i}^{-1}}\cdot\max_{d+1\leq j\leq m}\norm{\Omega_{j}}
        \cdot \max\limits_{d+1\leq j\leq m}\norm{\Omega_{j}}\norm{\Omega_{j}^{-1}}, \\  
        G_{d} :&= \max_{\substack{1\leq k\leq d\\ \lambda\in [\lambda_{\min},\lambda_{\max}]\setminus [\lambda_{\min}^{\cl},\lambda_{\max}^{\cl}]}}\bignorm{\fp_{k}(\lambda)},
    \end{aligned}
\end{equation*}
where $B_{k}=\Omega_{k}^{-1}\Lambda_{k}\Omega_{k}$ and $\fp_{k}$ is defined in \cref{eq:deffp} for $k=1,\dotsc,d$. Then, with probability one, 
\begin{equation}
    \label{eq:mainsb}
    \tan\angle\bigl(\range(Q),\mathcal{K}_{d}(A,\Omega)\bigr) \leq c_{\Omega}\cdot G_{d}.
\end{equation}
Furthermore, with probability at least $1-4\delta$, it holds that 
\begin{equation*}
    c_{\Omega}^{2}\leq  \frac{(dn-bd^{2})^{3}}{\delta^{4}}\Bigl(2\sqrt{b}+\sqrt{2\log\bigl(2(m-d)/\delta\bigr)}\Bigr).
\end{equation*}
\end{theorem}

\begin{proof}
    Recall the tail bounds for the smallest and largest singular value of Gaussian random matrices in \cite[Thm.~2.6 and Thm.~3.1]{Rudelson2010}:
    \begin{equation*}
        \Pr\bigl(\norm{\Omega_{0}}\leq 2\sqrt{b}+\sqrt{2\log(\delta/2)}\bigr) \geq 1-\delta\quad\text{and}\quad \Pr\bigl(\norm{\Omega_{0}^{-1}}\leq \sqrt{b}/\delta\bigr) \geq 1-\delta,
    \end{equation*}
    where $\Omega_{0}$ is a $b\times b$ Gaussian random matrix.
    Since all $\Omega_{i}$ and $\Omega_{j}$ are $b\times b$ Gaussian random matrices for $1\leq i\leq d<j\leq m$, the probabilistic result for $c_{\Omega}$ comes directly from taking union bounds on the tail bounds above.
    
    In the rest of the proof, we will focus on \cref{eq:mainsb}. Let $
    K = [Q^{\Ttran}\Omega,\dotsc,\Lambda^{d-1} Q^{\Ttran}\Omega]$ and $
    K_{\perp} = [Q_{\perp}^{\Ttran}\Omega,\dotsc,\Lambda_{\perp}^{d-1} Q_{\perp}^{\Ttran}\Omega]$.
    By \cite[Thm.~3.1]{Zhu2013}, we know that 
\begin{equation}
    \label{eq:pfthmmain1}
    \tan\angle\bigl(\range(Q),\mathcal{K}_{d}(A,\Omega)\bigr) = \norm{K_{\perp}K^{-1}}.
\end{equation} 
Since $\Omega$ is Gaussian, with probability one, all $\Omega_{i}$ are nonsingular.
By definition,  
\begin{equation*}
    K = \begin{bmatrix}
        \Omega_{1} & \Lambda_{1}\Omega_{1} & \cdots & \Lambda_{1}^{d-1}\Omega_{1}\\ 
        \Omega_{2} & \Lambda_{2}\Omega_{2} & \cdots & \Lambda_{2}^{d-1}\Omega_{2}\\ 
        \vdots & \vdots & \ddots & \vdots \\ 
        \Omega_{d} & \Lambda_{d}\Omega_{d} & \cdots & \Lambda_{d}^{d-1}\Omega_{d}
    \end{bmatrix}
    = D
    \begin{bmatrix}
        I & \Omega_{1}^{-1}\Lambda_{1}\Omega_{1} & \cdots & \Omega_{1}^{-1}\Lambda_{1}^{d-1}\Omega_{1}\\ 
        I & \Omega_{2}^{-1}\Lambda_{2}\Omega_{2} & \cdots & \Omega_{2}^{-1}\Lambda_{2}^{d-1}\Omega_{2}\\ 
        \vdots & \vdots & \ddots & \vdots \\ 
        I & \Omega_{d}^{-1}\Lambda_{d}\Omega_{d} & \cdots & \Omega_{d}^{-1}\Lambda_{d}^{d-1}\Omega_{d}
    \end{bmatrix},
\end{equation*}
where $D=\diagM{\Omega_{1},\Omega_{2},\dotsc,\Omega_{d}}$.
Recall that $B_{i} = \Omega_{i}^{-1}\Lambda_{i}\Omega_{i}$, we know $K = D\cdot\van$.
Similarly, denote $D_{\perp}=\diagM{\Omega_{d+1},\dotsc,\Omega_{m}}$ and $\Lambda_{\perp}=\diagm\{\lambda_{d+1},\dotsc,\Lambda_{m}\}$ with $\Lambda_{j}\in\R^{b\times b}$ for $d+1\leq j\leq m$, and let $B_{j}=\Omega_{j}^{-1}\Lambda_{j}\Omega_{j}$ and 
\begin{equation*}
    \van_{\perp} \defi
    \begin{bmatrix}
        I & B_{d+1} & \cdots & B_{d+1}^{d-1}\\ 
        I & B_{d+2} & \cdots & B_{d+2}^{d-1}\\ 
        \vdots & \vdots & \ddots & \vdots \\ 
        I & B_{m} & \cdots & B_{m}^{d-1}\\ 
    \end{bmatrix} \in\R^{bm\times bd},
\end{equation*}
it holds that $K_{\perp} = D_{\perp}\cdot\van_{\perp}$.   
In turn, by submultiplicativity of spectral norm, 
\begin{equation}
    \label{eq:pfthmmain2}
    \norm{K_{\perp}K^{-1}} \leq \norm{D_{\perp}}\norm{D^{-1}}\norm{\van_{\perp}\cdot\van^{-1}} \leq \norm{\van_{\perp}\cdot\van^{-1}} \max_{1\leq i\leq d}\norm{\Omega_{i}^{-1}}\max_{d+1\leq j\leq m}\norm{\Omega_{j}},
\end{equation}
where we use \cref{prop:nonsingular} to ensure the invertibility of $\van$.

In order to handle $\norm{\van_{\perp}\cdot\van^{-1}}$, we need to employ the matrix polynomial techniques developed in \Cref{sec:mplm}. 
By the definition of spectral norm, 
\begin{equation*}
    \begin{aligned}
        \norm{\van_{\perp}\cdot\van^{-1}} &= \max_{C^{\Ttran}C=I_{b}}\norm{\van_{\perp}\cdot\van^{-1}\cdot C} = \max_{(\van\cdot C)^{\Ttran}(\van\cdot C)=I_{b}}\norm{\van_{\perp}\cdot C}\\ 
        &= \max_{(\van\cdot C)^{\Ttran}(\van\cdot C)=I_{b}}\frac{\norm{\van_{\perp}\cdot C}}{\norm{\van\cdot C}} \leq \max_{\rk(C)=b}\frac{\norm{\van_{\perp}\cdot C}}{\norm{\van\cdot C}}.
    \end{aligned}
\end{equation*}
For any column full-rank matrix $C\in\R^{bd\times b}$ with the partition $C^{\Ttran}=[C_{0}^{\Ttran},\dotsc,C_{d-1}^{\Ttran}]$, where $C_{i}\in\R^{b\times b}$, we define a matrix polynomial $\Phi\in\poly_{d-1}(\R^{b\times b})$ as 
\begin{equation*}
    \Phi(X) = C_{0}+XC_{1}+\dotsb+X^{d-1}C_{d-1} \quad \forall\,X\in\R^{b\times b}.
\end{equation*} 
Then, by definitions of $\van$ and $\van_{\perp}$, we know 
\begin{equation*}
    \van\cdot C = 
    [   \Phi^{\Ttran}(B_{1}),\dotsc,
        \Phi^{\Ttran}(B_{d})]^{\Ttran}
    \quad \text{and}\quad 
    \van_{\perp}\cdot C = 
        [\Phi^{\Ttran}(B_{d+1}),\dotsc,
        \Phi^{\Ttran}(B_{m})]^{\Ttran}.
\end{equation*}
Then the spectral norm is bounded by 
\begin{equation}
    \label{eq:pfthmmain3}
    \norm{\van_{\perp}\cdot \van^{-1}} \leq \max_{\Phi\in\poly_{d-1}(\R^{b\times b})} \frac{\sqrt{m-d}\max\limits_{d+1\leq j\leq m}\norm{\Phi(B_{j})}}
    {\bignorm{[\Phi^{\Ttran}(B_{1}) 
        ,\dotsc,
        \Phi^{\Ttran}(B_{d})]^{\Ttran}}}.
\end{equation}
Note that all eigenvalues of $B_{j}$ for $d+1\leq j\leq m$ are in $[\lambda_{\min},\lambda_{\max}]\setminus [\lambda_{\min}^{\cl},\lambda_{\max}^{\cl}]$, \cref{lem:boundMP} yields that  
\begin{equation*}
    \max\limits_{d+1\leq j\leq m}\norm{\Phi(B_{j})}\leq \sqrt{b}\max_{d+1\leq j\leq m} \norm{\Omega_{j}}\norm{\Omega_{j}^{-1}} \max_{\lambda\in [\lambda_{\min},\lambda_{\max}]\setminus [\lambda_{\min}^{\cl},\lambda_{\max}^{\cl}]} \norm{\Phi(\lambda)}.
\end{equation*}
Recall that the assumption $\gap_{b}>0$ ensures that $\Lambda_{i}$ and $\Lambda_{j}$ have disjoint spectra for $i\neq j$, allowing us to use \cref{prop:wellpose} to claim that the expression of fundamental matrix polynomials in \cref{eq:deffp} are well-defined with probability one for $1\leq k\leq d$. Then we can apply \cref{thm:interpolation} to obtain that 
\begin{equation}
    \label{eq:pfthmmain4}
    \max_{\lambda\in [\lambda_{\min},\lambda_{\max}]\setminus [\lambda_{\min}^{\cl},\lambda_{\max}^{\cl}]} \norm{\Phi(\lambda)}
    \leq \sqrt{d}\cdot
        \bignorm{[\Phi^{\Ttran}(B_{1}) ,\dotsc,\Phi^{\Ttran}(B_{d})]^{\Ttran}} \cdot
        G_{d}.
    \end{equation}
    Combining \cref{eq:pfthmmain1,eq:pfthmmain2,eq:pfthmmain3,eq:pfthmmain4}, we conclude \cref{eq:mainsb}.
\end{proof}

The assumption $\gap_{b}>0$ essentially requires that the multiplicity of desired eigenvalues does not exceed $b$, which is, in fact, a necessary condition for the convergence of RSBL. To see this, suppose there exists an eigenvalue $\widetilde{\lambda}$ with multiplicity larger than $b$, and let $\widetilde{Q}$ be an orthonormal basis of the associated eigenspace. It holds that 
\begin{equation*}
    \dim \bigl( \range(\widetilde{Q})\cap \mathcal{K}_{\ell}(A,\Omega)\bigr) = 
    \dim \range\bigl(\widetilde{Q}\widetilde{Q}^{\Ttran}[\Omega,A\Omega,\dotsc,A^{\ell-1}\Omega]\bigr) = \dim \range(\widetilde{Q}\widetilde{Q}^{\Ttran}\Omega)\leq b
\end{equation*}
for any $\ell\geq 1$, implying that the Krylov subspace $\mathcal{K}_{\ell}(A,\Omega)$ contains an at most $b$-dimensional subspace of the eigenspace corresponding to $\widetilde{\lambda}$. Thus, RSBL can not find all eigenvectors corresponding to $\widetilde{\lambda}$. In particular, when $\widetilde{\lambda}=0$, this example reduces to the null space computation and explains why common wisdom \cite{Parlett1998} suggests employing a large-block Lanczos method.

Compared with the cluster robustness result for the single-vector Lanczos method in \cite[Eq.~(3.5)]{kressner2024randomized}, the dependence of $\delta$ in $c_{\Omega}$ is less favorable. This is due to the submultiplicativity of matrix norm used in \cref{lem:boundMP}, which does not arise in the single-vector case since all scalar products commute.

Another remark is that \cref{thm:sb} still applies when the number of eigenvalues of interest is not a multiple of the block size $b$, such as computing 10 eigenvectors with $b=3$. We can simply augment $\range(Q)$ with a few additional eigenvectors in $Q_{\perp}$ so that its dimension becomes a multiple of $b$. 

\subsection{Conjecture on cluster robustness}

In \cref{thm:sb}, the effect of spectral gaps in $[\lambda_{\min}^{\cl},\lambda_{\max}^{\cl}]$ is reflected in $G_{d}$. 
With notation in \cref{eq:defchi}, the inequality \cref{eq:Gd} yields that  
\begin{equation*}
    G_{d} \leq \Bigl(\frac{\chim\chic}{\gap_{b}}\Bigr)^{d-1}.
\end{equation*}
In order to further quantify the cluster robustness, we make the following conjecture.
\begin{conjecture}
    \label{conj:chi}
    Let $\Lambda_{1},\dotsc,\Lambda_{d}\in\R^{b\times b}$ be real diagonal matrices. Assume that $\Lambda_{i}$ and $\Lambda_{j}$ have disjoint spectra for $1\leq i<j\leq d$. Let $[\Omega_{1},\dotsc,\Omega_{d}]$ be a Gaussian random matrix where $\Omega_{i}\in\R^{b\times b}$. With high probability, the quantity $\chim\chic$ defined in \cref{eq:defchi} is bounded by a constant depending on $b$ and $d$, but independent of $\Lambda_{i}$ for $1\leq i\leq d$. In particular, this constant remains independent of the eigenvalue gaps both within and between the sets $\Lambda_{i}$.
\end{conjecture}

Under \cref{conj:chi}, we arrive at the following corollary on the cluster robustness of RSBL. We remark that this conjecture is consistent with the result for randomized single-vector ($b=1$) Lanczos method in \cite[Eq.~(3.5)]{kressner2024randomized}.

\begin{corollaryC}
    \label{conj:sb}
    Let $A$ be a symmetric matrix with the decomposition \cref{eq:decompA}, and $\Omega\in\R^{n\times b}$ be a Gaussian random matrix. Assume that $\gap_{b}$ in \cref{eq:defgap} is positive, then  
    \begin{equation*}
        \tan\angle\bigl(\range(Q),\mathcal{K}_{d}(A,\Omega)\bigr) \leq \Bigl(\frac{c_{b,d,n}}{\gap_{b}}\Bigr)^{d-1}
    \end{equation*}
    holds with high probability, where $c_{b,d,n}$ depends on $b$, $d$ and $n$.
\end{corollaryC}

\begin{remark}
    \label{rmk:sb}
    In both \cref{thm:sb,conj:sb}, we only assume that the eigenvalues in $\Lambda$ lie within the interval $[\lambda_{\min}^{\cl}, \lambda_{\max}^{\cl}]$, while the eigenvalues in $\Lambda_{\perp}$ lie outside this interval. Notably, $\Lambda$ may consist of either a cluster of interior eigenvalues or a subset of exterior eigenvalues. Moreover, the upper bound on $\tan\angle(\range(Q),\mathcal{K}_{d}(A,\Omega))$ is independent of the gap between $\Lambda$ and $\Lambda_{\perp}$.
\end{remark} 

\subsubsection{Empirical experiments}
\label{sec:numexp}
We verify \cref{conj:chi} and \cref{conj:sb} through numerical experiments implemented in Matlab 2022b and executed on an AMD Ryzen 9 6900HX Processor (8 cores, 3.3---4.9 GHz) with 32GB of RAM.
Scripts to reproduce
numerical results are publicly available at \url{https://github.com/nShao678/Small-block-Lanczos}.

The first part of numerical experiments is the empirical behaviors of the cluster robustness \cref{eq:defangle}, that is, \cref{conj:sb}.
Since Gaussian random matrices are invariant under rotation, we only consider diagonal matrices $A$ with the following form:
\begin{equation*}
    A = \diagm\{\Lambda,\Lambda_{\perp}\}\in\R^{n\times n}, \quad \text{where}\quad 
    \Lambda = \diagm\{\Lambda_{1},\dotsc,\Lambda_{d}\},
\end{equation*}
and $\Lambda_{k}$ are $b\times b$ diagonal matrices for $1\leq k\leq d$. Given a Gaussian random initial $\Omega\in\R^{n\times b}$, we first compute an orthonormal basis $V_{d}$ for the block Krylov subspace $\mathcal{K}_{d}(A,\Omega)$ by performing block Lanczos process with full reorthogonalization. Then the angle is computed by $
    \tan\angle (\range(Q),\mathcal{K}_{d}(A,\Omega)) = \norm{V_{\perp}V^{-1}}$,
where $V$ and $V_{\perp}$ contain the first $b\cdot d$ rows and rest part of $V_{d}$, respectively. In the following experiments, we set $n=1000$ and $bd=60$. 
As mentioned in \cref{rmk:sb}, we consider two $\Lambda_{\perp}$ as follows such that $\Lambda$ contains exterior eigenvalues or interior eigenvalues:
\begin{equation*}
    \Lambda_{\perp} = \begin{cases}
        -I-\diagm\{1/940,2/940,\dotsc,1\}\\ 
        \diagm\bigl\{-I-\diagm\{1/470,\dotsc,1\}, 4I+\diagm\{1/470,\dotsc,1\}\bigr\}
    \end{cases} .
\end{equation*}
For eigenvalues in $\Lambda_{k}$ with $1\leq k\leq d$, we set them as uniformly spaced in the interval 
\begin{equation*}
    [\alpha k+\beta(k-1),\alpha k +\beta(k+1)]/60 \subset [0,1],   
\end{equation*}
where $\alpha$ and $\beta$ are positive parameters to be specified later. Note that in this situation, the cluster radius is $\beta/60$, and the absolute gap between nearby  clusters is $\alpha/60$.

To study how $\tan\angle (\range(Q),\mathcal{K}_{d}(A,\Omega))$ depends on key parameters, we conduct two experiments summarized in \cref{fig:EP}. In both, we collect empirical results from 1000 independent trials while varying $d$. In the first experiment, we fix $\alpha = 1$ and vary the cluster radius as $\beta = 2^{-i}$ for $i = 1,\dotsc,12$, with $d = 2,3,4,5$. The results show that $\tan\angle (\range(Q),\mathcal{K}_{d}(A,\Omega))$ is independent of $\beta$. In particular, this suggests that the cluster robustness of RSBL is not related to the eigenvalue gaps within each $\Lambda_{i}$, and hence, the standard eigenvalue gap $\gap_{1}$ required in \cite{Chen2026}.
In the second, we fix $\beta = 10^{-4}$ and vary the spectral gap via $\alpha = 2^{-i}$ for $i = 1,\dotsc,10$, with $d = 2,3,4$. The outcome demonstrates that for both exterior and interior eigenvalues, $\tan\angle (\range(Q),\mathcal{K}_{d}(A,\Omega))$ is proportional to $\gap_{b}^{1-d}$, consistent with \cref{conj:sb}.

\begin{figure}[htbp]
    \centering
    
    \includegraphics[width=\figsizeD]{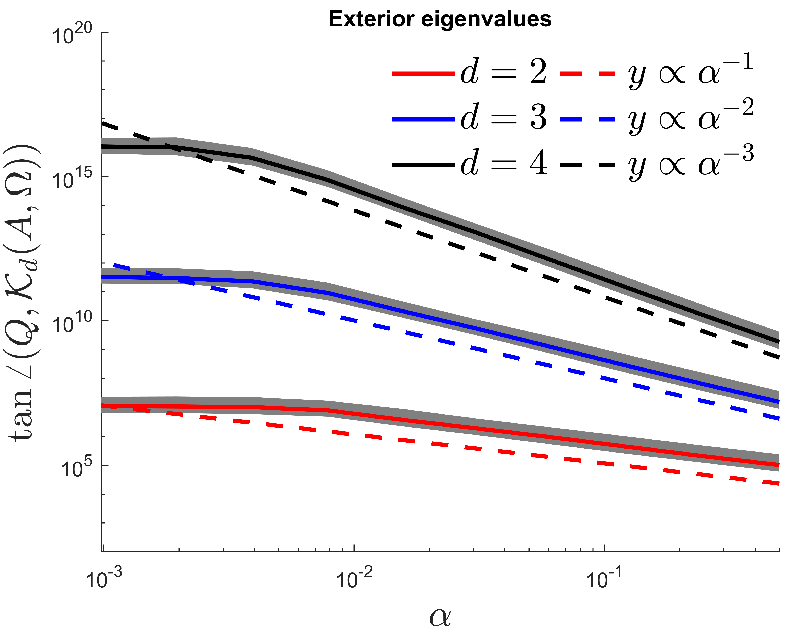}
    \includegraphics[width=\figsizeD]{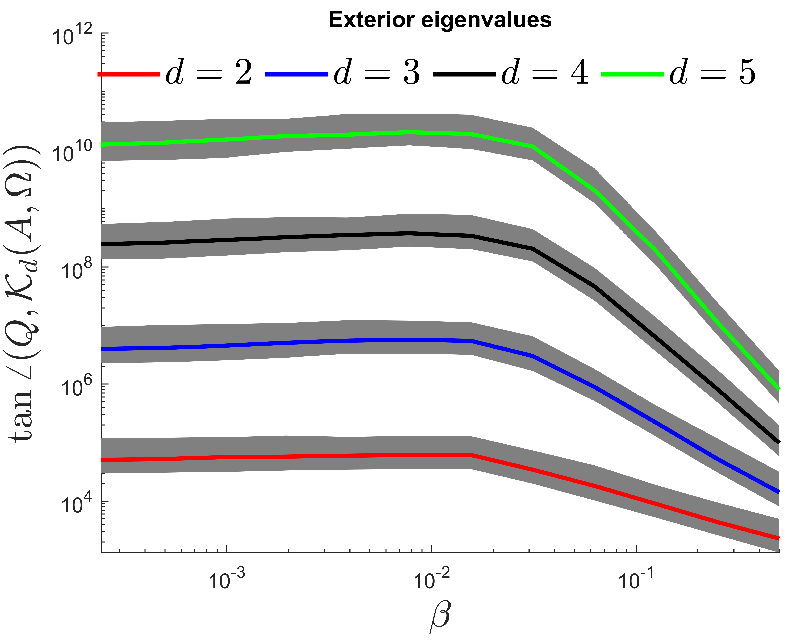}

    \includegraphics[width=\figsizeD]{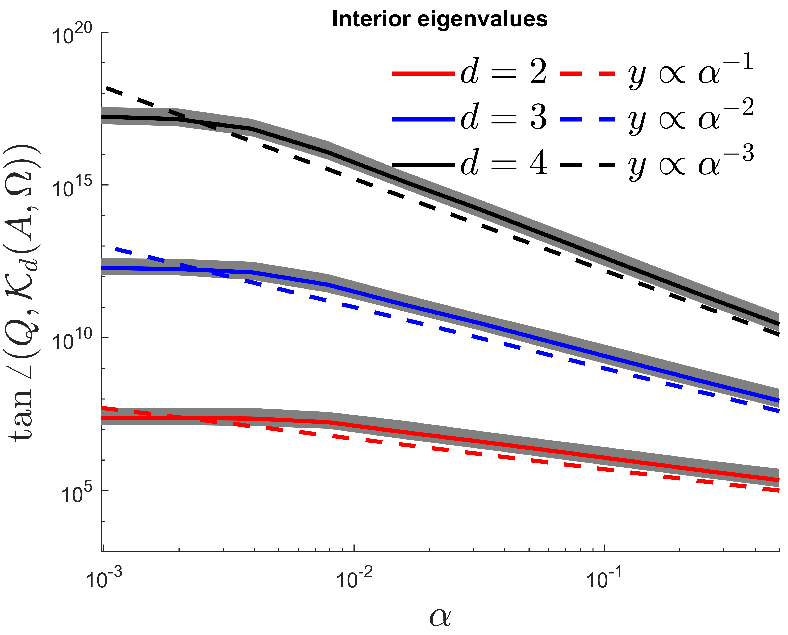}
    \includegraphics[width=\figsizeD]{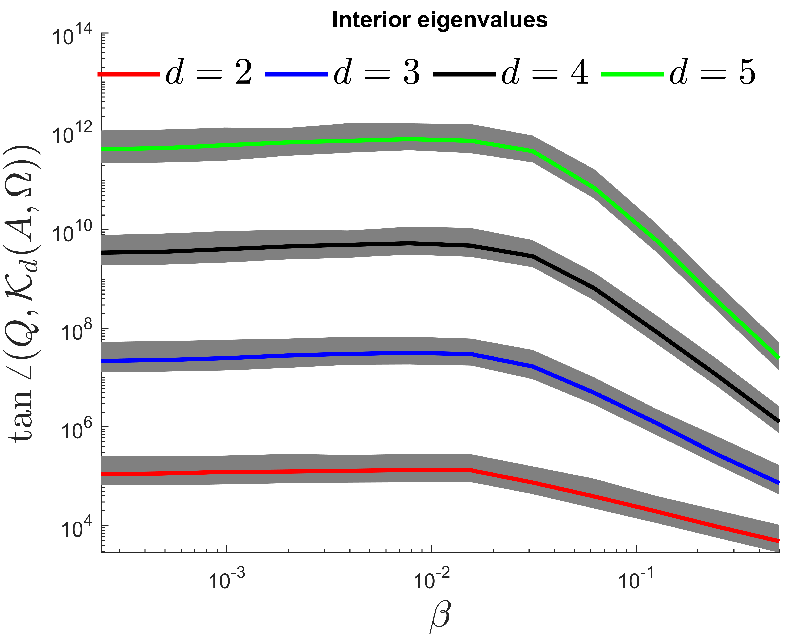}
    
    \caption{
        Cluster robustness for exterior eigenvalues (top) and interior eigenvalues (down) with different eigenvalue gaps. The left panels contain matrices with varying eigenvalue gaps between the clusters, whereas the right panels contain matrices with different eigenvalue gaps within each cluster. 
    The median of empirical $\tan\angle (\range(Q),\mathcal{K}_{d}(A,\Omega))$ over 1000 independent trials is plotted, with the 25th and 75th quartiles shaded in.}
    \label{fig:EP}
\end{figure}

\begin{remark}
    As already mentioned in \cref{rmk:intro}, although $\tan\angle (\range(Q),\mathcal{K}_{d}(A,\Omega))$ in \cref{fig:EP} appears quite large, this is indeed a regime in which RSBL performs effectively.
\end{remark}

The second part is about the empirical behavior of $\chim$ and $\chic$ conjectured in \cref{conj:chi}.
Note that these two parameters are not related to the rest eigenvalues in $\Lambda_{\perp}$.
Using the same setting as above, we collect empirical results in \cref{fig:EPchi}, suggesting that the upper bounds of these two quantities are independent of both gaps $\alpha$ between clusters and gaps $\beta$ within each clusters.
\begin{figure}[htbp]
    \centering
    \includegraphics[width=\figsizeD]{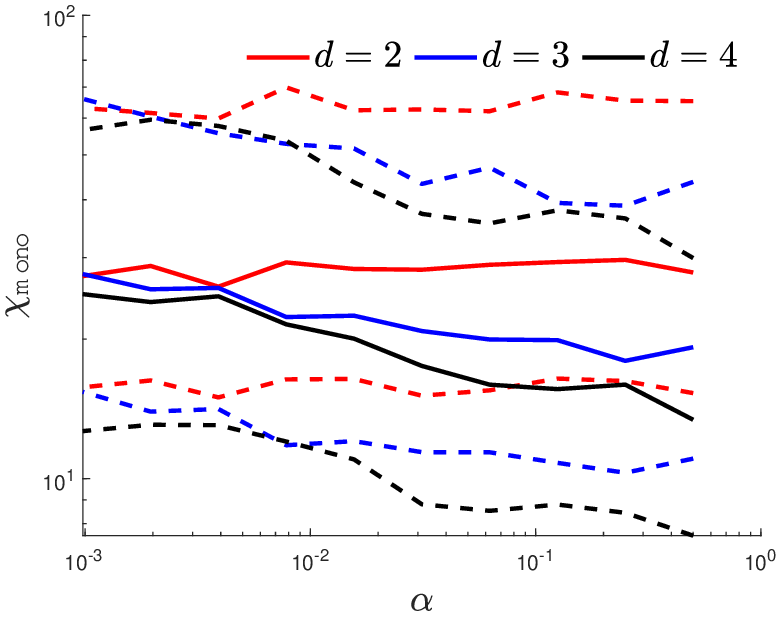}
    \includegraphics[width=\figsizeD]{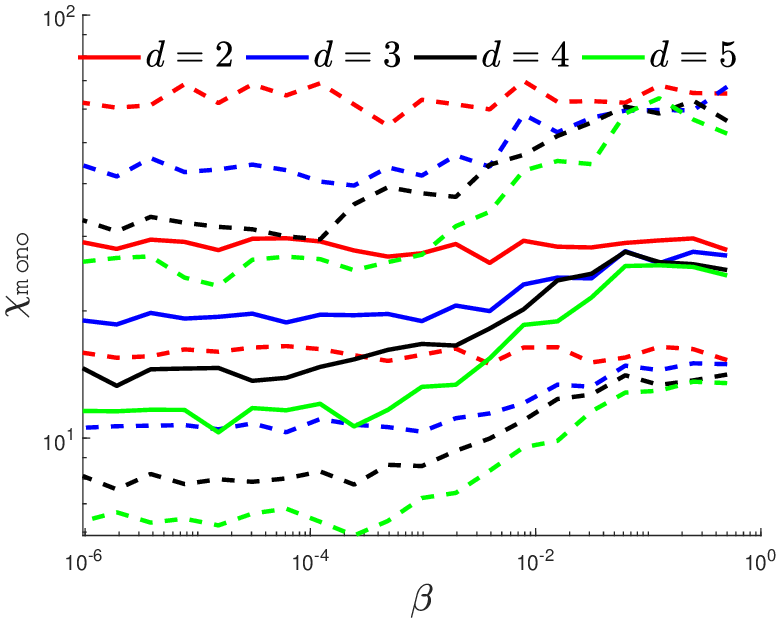}

    \includegraphics[width=\figsizeD]{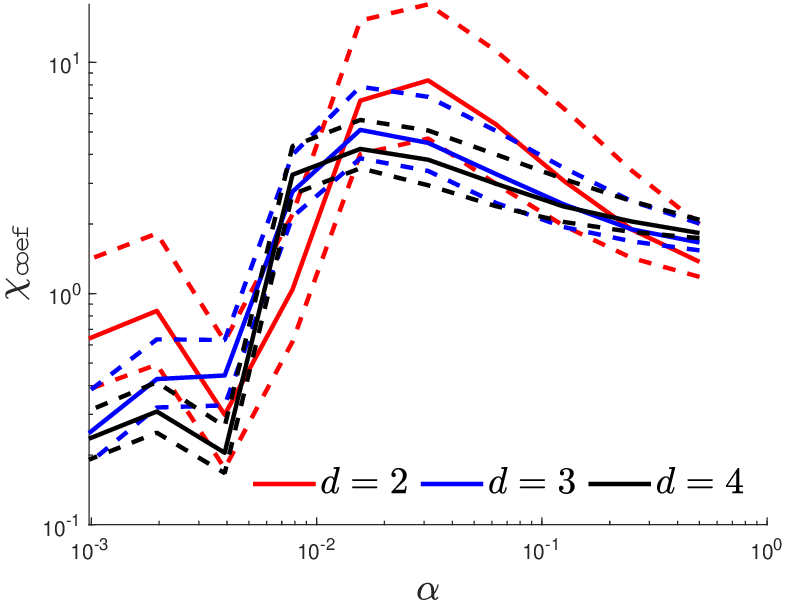}
    \includegraphics[width=\figsizeD]{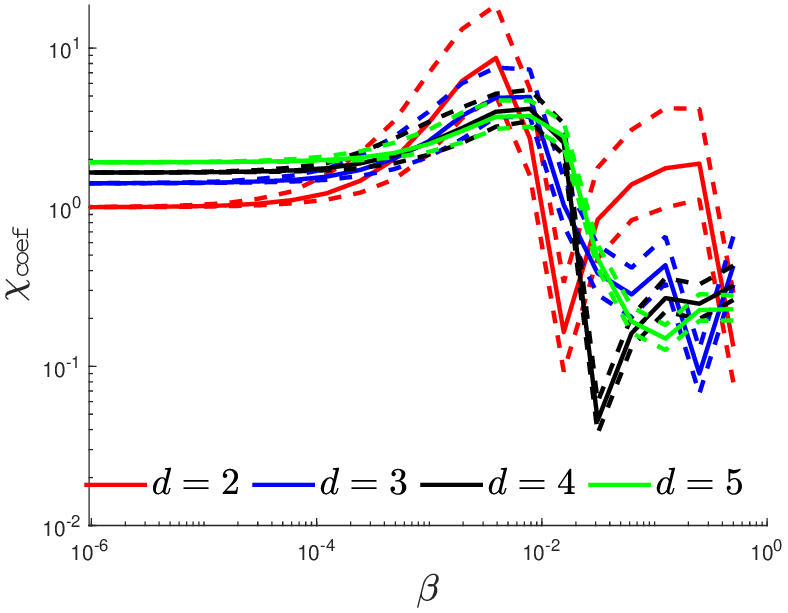}

    \caption{$\chim$ (top) and $\chic$ (bottom) with different eigenvalue gaps. The left panels contain matrices with varying eigenvalue gaps between the clusters, whereas the right panels contain matrices with different eigenvalue gaps within each cluster.
    The median of the empirical results over 1000 independent trials is plotted in solid lines, with the 25th and 75th quartiles in dashed lines.}
    \label{fig:EPchi}
\end{figure}

\subsubsection{Intrinsic challenge in analysis of RSBL}

Since both \cref{thm:sb} and \cref{conj:sb} are quite abstract, we illustrate them by working out the simplest non-trivial case: $d = 2$. Now, recalling that $\widehat{B}_{2} =B_{2} = \Omega_{2}^{-1} \Lambda_{2} \Omega_{2}$ and $\widehat{B}_{2}^{(2)} = B_{1} = \Omega_{1}^{-1} \Lambda_{1} \Omega_{1}$, we know
\begin{equation*}
    \begin{aligned}
        \chim &=  \max_{i=1,2} \max\Bigl\{ \frac{\norm{\lambda_{\min}^{\cl}I-B_{i}}}{\norm{\lambda_{\min}^{\cl}I-\Lambda_{i}}},\frac{\norm{\lambda_{\max}^{\cl}I-B_{i}}}{\norm{\lambda_{\max}^{\cl}I-\Lambda_{i}}},1\Bigr\}\leq\max_{i=1,2} \norm{\Omega_{i}}\norm{\Omega_{i}^{-1}},\\ 
    \chic &= \norm{(B_{1}-B_{2})^{-1}}\min_{\lambda_{1}\in \Lambda_{1}, \lambda_{2}\in \Lambda_{2}} \abs{\lambda_{1}-\lambda_{2}}.
    \end{aligned}
\end{equation*}
In this simple situation, $\chim$ can be bounded using existing results from random matrix theory; see~\cite{Rudelson2010}, for example. However, estimating $\chic$, or equivalently $\norm{(B_{1}-B_{2})^{-1}}$, is highly non-trivial since both $B_{1}$ and $B_{2}$ are \emph{nonsymmetric}, although its empirical behavior is good in \cref{fig:EPchi}.
In the following, we demonstrate that estimating $\norm{(B_{1}-B_{2})^{-1}}$ constitutes an intrinsic difficulty of RSBL. We first establish an equivalence between $\norm{(B_{1}-B_{2})^{-1}}$ and the smallest singular value of the block Vandermonde matrix $\van = \begin{bmatrix}
        I & B_{1} \\ 
        I & B_{2}
    \end{bmatrix}$. 
\begin{lemma}
    \label{lem:sb}
    Suppose that $B_{1}-B_{2}$ is nonsingular, then 
    \begin{equation*}
        \frac{3-\sqrt{5}}{2}\norm{(B_{1}-B_{2})^{-1}}^{2}\leq 
        \Bignorm{\begin{bmatrix}
        I & B_{1} \\ 
        I & B_{2}
    \end{bmatrix}^{-1}}^{2} \leq \frac{3+\sqrt{5}}{2}\bigl(2+(2\norm{B_{1}}^{2}+1)\norm{(B_{1}-B_{2})^{-1}}^{2}\bigr).
    \end{equation*}
\end{lemma}

\begin{proof}
Let $X=-B_{1}$ and $Y=(B_{2}-B_{1})^{-1}$. The block LU factorization yields  
\begin{equation}
    \label{eq:appLU}
    \begin{bmatrix}
        I & B_{1} \\ 
        I & B_{2}
    \end{bmatrix}^{-1}
    = 
    \begin{bmatrix}
        I & XY \\  & Y
    \end{bmatrix}
    \begin{bmatrix}
        I & \\ -I & I
    \end{bmatrix}.
\end{equation}
We just need to bound the block upper triangular matrix in \cref{eq:appLU}.
On the one hand,
\begin{equation*}
    \norm{Y} = \max_{\norm{x}=1}\norm{Yx} \leq \max_{\norm{x}=1}\Bignorm{\begin{bmatrix}
        XYx\\ Yx
    \end{bmatrix}}  = \max_{\norm{x}=1}\Bignorm{\begin{bmatrix}
        I & XY \\ 
        & Y
    \end{bmatrix} \begin{bmatrix}
        0\\ x
    \end{bmatrix}} \leq \Bignorm{\begin{bmatrix}
        I & XY \\ 
        & Y
    \end{bmatrix} }.
\end{equation*} 
On the other hand, 
\begin{equation*}
    \begin{aligned}
        \Bignorm{\begin{bmatrix}
        I & XY \\ 
        & Y
    \end{bmatrix}}^{2} 
    &= \max_{\norm{x}^{2}+\norm{y}^{2}=1}
    \Bignorm{\begin{bmatrix}
        I & XY \\ 
        & Y
    \end{bmatrix}\begin{bmatrix}
        x\\ y
    \end{bmatrix}}^{2}
    = \max_{\norm{x}^{2}+\norm{y}^{2}=1}\norm{x+XYy}^{2}+\norm{Yy}^{2} \\ 
    & \leq \max_{\norm{x}^{2}+\norm{y}^{2}=1}  \norm{x}^{2}+ 2\norm{X}\norm{Y}\norm{x}\norm{y}+(\norm{X}^{2}+1)\norm{Y}^{2}\norm{y}^{2}
    \leq 2+(2\norm{X}^{2}+1)\norm{Y}^{2}.       
    \end{aligned}
\end{equation*}
Combining these two inequalities above with the singular values of the lower triangular matrix in \cref{eq:appLU}, we finish the proof. 
\end{proof}

The smallest singular value of the block Vandermonde matrix is, in fact, a fundamental measure of the cluster robustness. To illustrate this, we let
\begin{equation*}
    A_{2b} = \diagm\{1,\dotsc,2b\}\in\R^{2b\times 2b}
    \quad\text{and}\quad 
    A = \diagm\{A_{2b},3bI_{b},4bI_{b},\dotsc,mI_{b}\}\in\R^{n\times n}.
\end{equation*}
Recall the proof for \cref{thm:sb}, it holds that  
\begin{equation*}
    \tan\angle(\range(Q),\mathcal{K}_{d}(A,\Omega)) = D_{\perp}\cdot \van_{\perp}\cdot \van^{-1}\cdot D^{-1},
\end{equation*}
where $D_{\perp}$ and $D$ are both block diagonal matrices with independent Gaussian blocks. By tail bounds of singular values for Gaussian random matrices \cite{Rudelson2010}, we know that $\tan\angle(\range(Q),\mathcal{K}_{d}(A,\Omega))$ and $\norm{\van_{\perp}\cdot\van^{-1}}$ are equivalent (up to a low-degree polynomial of $n$) with high probability. Note that $B_{k}=\Omega_{k}^{-1}\Lambda_{k}\Omega_{k}=kbI_{b}$ for $3\leq k\leq m$ with probability one, which yields that 
\begin{equation*}
    \van_{\perp} = 
    \begin{bmatrix}
        I_{b} & B_{3}\\ 
        \vdots & \vdots\\ 
        I_{b} & B_{m}
    \end{bmatrix}
    =  \begin{bmatrix}
        I_{b} & 3bI_{b} \\ 
        \vdots & \vdots \\ 
        I_{b} & mbI_{b}
    \end{bmatrix}.
\end{equation*}
Consequently, $\norm{\van^{-1}}$ is equivalent to $\norm{\van_{\perp}\cdot\van^{-1}}$ with probability one, and thus also equivalent to $\tan\angle(\range(Q),\mathcal{K}_{d}(A,\Omega))$ with high probability. 

Therefore, estimating $\norm{(B_{1}-B_{2})^{-1}}$ is as challenging  as analyzing RSBL. For more general situations $d> 2$,  terms like this are in hidden in $\norm{S_{1,d}^{-1}}$ in \cref{lem:Chain}, and $\chim$ also becomes non-trivial due to the presence of $\widehat{B}_{i}$. Mathematically, we have to handle the smallest singular of the following operator: 
\begin{equation}
    \label{eq:intrinsic}
    \Lambda_{i}\Omega_{i}-\Omega_{i}\widetilde{\Omega}_{j}^{-1}\Lambda_{j}\widetilde{\Omega}_{j} = \Omega_{i}\Bigl(\Omega_{i}^{-1}\Lambda_{i}\Omega_{i}-\widetilde{\Omega}_{j}^{-1}\Lambda_{j}\widetilde{\Omega}_{j}\Bigr)\in\R^{b\times b},
\end{equation} 
where $\Omega_{i}$ is a Gaussian random matrix, $\widetilde{\Omega}_{j}$ is a random matrix with potentially complicated distribution but independent of $\Omega_{i}$, and $\Lambda_{i}$ and $\Lambda_{j}$ are deterministic diagonal matrices with disjoint spectra. 
When $b=1$, RSBL becomes single-vector Lanczos method. This analysis becomes trivial since all scalar products \emph{commute} and in turn, $\Omega_{i}$ and $\widetilde{\Omega}_{j}$ are cancelled with $\Omega_{i}^{-1}$ and $\widetilde{\Omega}_{j}^{-1}$, respectively. 
Another relatively simple scenario occurs when $\widetilde{\Omega}_{j}$ is an orthogonal matrix,  reduces the problem to estimating the smallest singular value of $\Lambda_{i,j}\circ \Omega_{i}$, where $\circ$ denotes the Hadamard product, and $\Lambda_{i,j}$ is a deterministic matrix containing the differences of eigenvalues between $\Lambda_{i}$ and $\Lambda_{j}$. Thanks to the independence of entries in $\Lambda_{i,j}\circ \Omega_{i}$, its smallest singular value can be bounded via the so-called broad connectivity; see, \cite{Rudelson2016,Cook2018} for example. However, when $\widetilde{\Omega}_{j}$ is a non-orthogonal matrix, the nonsymmetry of $\widetilde{\Omega}_{j}^{-1}\Lambda_{j}\widetilde{\Omega}_{j}$ breaks the independence of entries in $\Lambda_{i,j}\circ \Omega_{i}$, causing existing arguments to fail. 
Thus, estimating the smallest singular value of the matrix in \cref{eq:intrinsic} presents an intrinsic challenge in the analysis of the RSBL.

\subsubsection{Applications}
\label{sec:app}
\paragraph{Eigenvalue computation:}
Suppose we are interested in the largest $b\cdot d$ eigenvalues of a symmetric matrix~$A$. Assume that the relative gap defined in \cref{eq:defgap} is positive, and $\lambda_{bd}-\lambda_{bd+1}>0$, where $\lambda_{bd}$ and $\lambda_{bd+1}$ are the largest $(b\cdot d)$th and $(bd+1)$st eigenvalues of $A$, respectively. Applying \cref{conj:sb} to \cite[Lem.~3.4]{kressner2024randomized} with some shift of indices, we know that, with high probability,
\begin{equation}
    \label{cor:eig}
        \tan\angle\bigl(\range(Q),\mathcal{K}_{\ell}(A,\Omega)\bigr) \leq \frac{\bigl(c_{b,d,n}/\gap_{b}\bigr)^{d-1}}{\cheb_{\ell-d}(1+
        2\frac{\lambda_{bd}-\lambda_{bd+1}}{\lambda_{\max}-\lambda_{\min}})},
\end{equation}
where $Q$ contains the first $b\cdot d$ eigenvectors, $c_{b,d,n}$ comes from \cref{conj:sb} and $\cheb_{\ell-d}$ is the Chebyshev polynomial of degree $(\ell-d)$ defined as 
\begin{equation*}
    \cheb_{\ell-d}(\lambda) = \frac{1}{2}\Bigl(\bigl(\lambda+\sqrt{\lambda^{2}-1}\bigr)^{\ell-d}+\bigl(\lambda+\sqrt{\lambda^{2}-1}\bigl)^{d-\ell}\Bigr).
\end{equation*}

When $d = 1$, \cref{cor:eig} recovers the convergence result for randomized single-vector Lanczos method from \cite[Thm.~3.1]{kressner2024randomized}. Beyond enabling a potentially larger gap parameter $\gap_{b}$, choosing a block size $b > 1$ can substantially reduce the exponent of $\gap_{b}^{-1}$, thereby enhancing robustness to eigenvalue clusters. Although the number of matvecs required to achieve a polynomial acceleration of degree $\ell - d$ increases from $\ell$ to $b\ell$, the overall complexity remains comparable when $\gap_{b}$ is small. In practice, the small-block variant can even outperform its single-vector counterpart due to lower communication overhead.

\paragraph{Low-rank approximation:}
Given a matrix $\widehat{A} \in\R^{N\times n}$ with $N\geq n$, one can find a good rank-$bd$ approximation by running Lanczos for $A=\widehat{A}^{\Ttran}\widehat{A}$. Different from eigenvalue problems, finding top $b\cdot d$ eigenvectors of $A$ is sufficient but not necessary for obtaining a good low-rank approximation \cite{Drineas2019}. 
Denote $V$ as the top $b\cdot d$ Ritz vectors extracted from the (block) Krylov subspace $\mathcal{K}_{\ell}(A,\Omega)$.
Plugging \cref{conj:sb} into \cite[Thm.~3.1]{Meyer2023}, we know that when 
\begin{equation}
    \label{cor:lr}
    \ell \geq  \frac{c}{\sqrt{\epsilon}}\bigl(\log\frac{1}{\epsilon}+d\log\frac{c_{b,d,n}}{\gap_{b}}\bigr),
\end{equation}
we can extract a low-rank approximation from $\mathcal{K}_{\ell}(A,\Omega)$ satisfying
\begin{equation*}
    \norm{\widehat{A}-\widehat{A}VV^{\Ttran}} \leq (1+\epsilon) \norm{\widehat{A}-\widehat{A}_{bd}} 
    \quad \text{and}\quad
    \norm{\widehat{A}-\widehat{A}VV^{\Ttran}}_{\fro} \leq (1+\epsilon) \norm{\widehat{A}-\widehat{A}_{bd}}_{\fro}
\end{equation*}
with high probability, where $c>0$ is an absolute constant, $\widehat{A}_{bd}$ is the best rank-$bd$ approximation of $\widehat{A}$. Moreover, let $v_{i}$ be the $i$th Ritz vector for $1\leq i\leq bd$, then 
\begin{equation*}
    \bigabs{\norm{\widehat{A}v_{i}}^{2}-\lambda_{i}}\leq \epsilon \lambda_{bd+1},
\end{equation*}
where $\lambda_{i}$ and $\lambda_{bd+1}$ are the $i$th and $(bd+1)$st  largest eigenvalues of $A$, respectively.
    
Compared with \cite[Thm.~4.5]{Meyer2023}, one $b$ factor is removed in the coefficient of $\log(1/\gap_{b})$ in \cref{cor:lr}. As a consequence, it achieves the conjecture ``the matrix-vector complexity scales linearly in the target rank for any block size'' in \cite[p.~823]{Meyer2023}.

\section{Conclusion}

In this paper, we presented a structural bound in \cref{thm:sb} along with a conjectured probabilistic bound to quantify the cluster robustness in RSBL for eigenvalue computation. The conjecture offers insights into the behavior of a class of randomized small-block eigensolvers and provides a foundation for further theoretical exploration. Future work will focus on establishing a rigorous probabilistic bound to support the conjectured results. 
This may involve the recently developed comparison theorems \cite{tropp2025comparison,Brailovskaya2024}.

\section*{Acknowledgments}
The author thanks Daniel Kressner for insightful discussions and constructive comments that improved this work. He also thanks Ethan N. Epperly and Tyler Chen for their helpful comments, as well as the two anonymous referees, whose critical insights significantly enhanced the presentation.

\bibliographystyle{abbrvurl}

\end{document}